\documentclass[a4paper, two-side, 11pt]{amsart}
\usepackage[margin= 1in]{geometry}    

\usepackage{tikz}
\usepackage{cancel}
\usepackage[colorlinks,linkcolor=blue,citecolor=blue,urlcolor=blue]{hyperref}
\usepackage[english]{babel}
\usepackage{amsthm,amsfonts,amssymb,mathrsfs, enumerate}
\usepackage{amsmath}
\usepackage{xcolor}
\usepackage[normalem]{ulem}
\usepackage{enumerate}
\usepackage{lmodern}
\usepackage{import}
\usepackage{mathtools}

\newcommand{\Cross}{\mathbin{\tikz [x=2ex,y=2.5ex,line width=.2ex] \draw (0,0) -- (1,1) (0,1) -- (1,0);}}%
\mathtoolsset{showonlyrefs,showmanualtags}	

\usepackage{todonotes}  

\newtheorem{theorem}{Theorem}[section]
\newtheorem{Coro}[theorem]{Corollary}
\newtheorem{Lemma}[theorem]{Lemma}
\newtheorem{Sublemma}[theorem]{Sublemma}
\newtheorem{Prop}[theorem]{Proposition}
\newtheorem{Conj}[theorem] {Conjecture}

\theoremstyle{definition}
\newtheorem{Def}[theorem]{Definition}

\theoremstyle{remark}
\newtheorem{rk}[theorem]{Remark}
\newtheorem{example}[theorem]{Example}

\theoremstyle{definition}

\theoremstyle{remark}

\renewcommand{\aa}{\alpha}

\newcommand{\mycomment}[1]{}

\newcommand{\R}{\mathbb{R}}

\newcommand{\N}{\mathbb{N}}

\numberwithin{equation}{section}

\title[On Eisenhart's type theorem for sub-Riemannian metrics]{On  Eisenhart's type theorem for sub-Riemannian metrics on step $2$ distributions with {\normalsize $\mathrm{ad}$}-surjective Tanaka symbols}
\date{\today}

\thanks{ I.\ Zelenko is partly supported by NSF grant DMS 2105528 and Simons Foundation Collaboration Grant for Mathematicians 524213.}

\author{Zaifeng Lin}
\address{Zaifeng Lin\\
	Department of Mathematics\\
	Texas A\&M University\\
	College Station\\
	Texas \ 77843\\
	USA}
\email{linzf@tamu.edu}

\author{Igor Zelenko}
\address{Igor Zelenko\\
         Department of Mathematics\\
         Texas A\&M University\\
         College Station\\
         Texas \ 77843\\
         USA}
\email{zelenko@math.tamu.edu}
\urladdr{\url{http://www.math.tamu.edu/~zelenko}}

\begin{document}
\begin{abstract}
The classical result of Eisenhart states that if a Riemannian metric $g$ admits a Riemannian metric that is not constantly proportional to $g$  and has the same (parameterized) geodesics as $g$ in a neighborhood of a given point,  then $g$ is a direct product of two Riemannian metrics in this neighborhood.  We introduce a new generic class of step $2$ graded nilpotent Lie algebras, called \emph{$\mathrm{ad}$-surjective}, and extend the Eisenhart theorem to sub-Riemannian metrics on step $2$ distributions with $\mathrm{ad}$-surjective Tanaka symbols.  The class of ad-surjective step $2$ nilpotent Lie algebras contains a well-known class of algebras of H-type as a very particular case. 
\end{abstract}

\maketitle
\section{Introduction}
\subsection{Affine equivalence in Riemannian geometry: nonrigidity and product structure}
 The paper is devoted to a problem in sub-Riemannian geometry but we start with a historical overview of the same problem in Riemannian geometry. Recall that 
two Riemannian metrics   $g_{1}$ and $g_{2}$ on a manifold $M$ are called projectively equivalent if they have the same geodesics, as unparametrized curves, namely,  for every geodesic $\gamma(t)$ of $g_1$ there exists a reparametrization $t=\varphi(\tau)$ such that $\gamma\bigl(\varphi(\tau)\bigr)$ is a geodesic of $g_2$.
They are called  affinely equivalent, 
if they are projective equivalent and the reparametrizations $\varphi(\tau)$ above are affine functions, i.e., they are of the form $\varphi(\tau)=a\tau+b$.
We will write $g_1\stackrel{p}{\sim} g_2$ and $g_1\stackrel{a}{\sim} g_2$ in the case of projective and affine equivalence, respectively. In the sequel, we will mainly be interested in the local version of the same definitions for germs of Riemannian metrics at a point when conditions on the coincidence of geodesics hold in a neighborhood of this point.

From the form of the equation for Riemannian geodesics, it follows immediately that two Riemannian metrics are affinely equivalent if and only if they have the same geodesics as parametrized curves, which in turn is equivalent to the condition that they have the same Levi-Civita connection, i.e., one metric is parallel with respect to the Levi-Civita connection of the other.

Obviously, given any Riemannian metric $g$  and a positive constant $c$,  the metrics $cg$ and $g$ are affinely equivalent. The metric $cg$ will be said a \emph{ constantly proportional metric} to the metric $g$. The Riemannian metric $g$ is called \emph{affinely rigid} if metrics constantly proportional to it are the only 
affinely equivalent metric to it. 

A class of Riemannian metrics $g$ that are not affinely rigid are the metrics admitting a product structure, i.e., when the ambient manifold $M$ can be represented as $M=M_1\times M_2$, where each $M_1$ and $M_2$ are of positive dimension, and there exist Riemannian metrics $g_1$ and $g_2$ on $M_1$ and $M_2$, respectively, such that if $\pi_i: M\rightarrow M_i$, $i=1,2$, are canonical projections, then
\begin{equation} 
\label{prodmetric}
g=\pi_1^*g_1+\pi_2^* g_2.
\end{equation}
Then, obviously, for every two positive constants $C_1$ and $C_2$ 
\begin{equation}
\label{C1C2} 
g\stackrel{a}\sim\left(C_1\pi_1^*g_1+C_2\pi_2^* g_2\right)
\end{equation}
and the metric $\left(C_1\pi_1^*g_1+C_2\pi_2^* g_2\right)$ is not constantly proportional to $g$ if $C_1\neq C_2$, i.e., the metric $g$ is not affinely rigid. In 1923 L. P.  Eisenhart proved that locally the converse is true, i.e., the following theorem holds

\begin{theorem} [\cite{Eisenhart1923}] 
\label{eisenthm} If a Riemannian metric $g$ is not affinely rigid near a point $q_0$, i.e., admits a locally affinely equivalent non-constantly proportional Riemannian metric in a neighborhood of a point $q_0$,  
	then the metric $g$ is the direct product of two Riemannian metrics in a neighborhood of $q_0$.
 \end{theorem}
This theorem is closely related to (and actually is a local version of) the De Rham decomposition theorem (\cite{deRham1952}) on the direct product structure of a simply connected complete Riemannian manifolds in terms of the decomposition of the tangent bundle into invariant subbundles with respect to the action of the holonomy group. Indeed, if $g_1 \stackrel{a}\sim g_2$
and these metrics are not constantly proportional, then 
the eigenspaces of the transition operators between these metrics (see \eqref{transition_op_eq} below for the definition of the transition operator) form such a decomposition of the tangent bundle  with respect to the action of the holonomy group of both  $g_1$ and $g_2$ (recall that they have the same Levi-Civita connection).
\subsection{Affine equivalence of sub-Riemannian metrics: the main conjecture}
First, recall that a distribution $D$ on a manifold $M$ is a subbundle of the tangent bundle $TM$. 
A sub-Riemannian manifold/structure is a triple $(M, D,g)$, where $M$ is a smooth manifold, $D$ is a  bracket-generating distribution, and for any $q$, $g(q)$ is an inner product on $D(q)$ which depends smoothly on $q$. We say that $g$ is a sub-Riemannian metric on $(M, D)$. A Riemannian manifold/structure/metric appears as the particular case where $D=TM$. 

In the sequel, we will assume that the distribution $D$ is bracket-generating, i.e. for every point $q\in M$ the iterative Lie brackets of vector fields tangent to a distribution $D$ (i.e., of sections of $D$) span the tangent space $T_qM$. In more details, 
one can define a filtration 
\begin{equation}
\label{filtr}
D=D^{1}\subset D^{2}\subset\ldots D^j\subset \ldots
\end{equation}
of the tangent bundle, called a \emph{weak derived flag}, as follows:  set $D=D^{1}$ and define recursively
\begin{equation}
\label{Dpower1}  
D^{j}=D^{j-1}+[D,D^{j-1}], \quad  j>1.
\end{equation}
If $X_1,\ldots, X_m$ are $m$ vector fields constituting a local basis of a distribution $D$,
then $D^{j}(q)$ is the linear span of all iterated Lie brackets of these vector fields, of length not greater than  $j$,  evaluated at a point $q$,
\begin{equation}
\label{Dpower2}
D^{j}(q)=  {\rm span}\{\left[X_{i_1}(q),\ldots [X_{i_{s-1}},X_{i_s}](q)\ldots\right]:  (i_1,\ldots,i_s)\in [1:m]^s, s\in[1:j] \}
\end{equation}
(here given a positive integer $n$ we denote by $[1:n]$ the set $\{1, \ldots, n\}$).
A distribution $D$ is called \emph{bracket-generating} (or \emph{completely nonholonomic}) if for any $q$ there exists $\mu(q)\in\mathbb N$ such that $D^{\mu(q)}(q)=T_q U$. The number $\mu(q)$ is called the \emph{degree of nonholonomy} of $D$ at a point $q$. If the degree of nonholonomy is equal to a constant $\mu$ at every point, one says that $D$ is \emph{step $\mu$ distribution}.

Since we work locally, the assumption of bracket-genericity is not too restrictive: if a distribution is not bracket-generating then in a neighborhood $U$ of a generic point there exists a positive integer $\mu$ such that $D^{\mu+1}=D^{\mu}\subsetneqq TM$. So,  $D^\mu$ is a proper involutive subbundle of $TU$ and the distribution $D$ is bracket-generating on each integral submanifold of  $D^\mu$ in $U$. So, we can restrict ourselves to these integral submanifolds instead of $U$.


What are sub-Riemannian geodesics? There are at least two different approaches to this concept. One approach is variational: a geodesic is seen as an extremal trajectory, i.e., a  candidate for the ``shortest" or ``energy-minimizing" path connecting its endpoints,  with respect to the corresponding length or energy functional. The other approach is differential-geometric: geodesic is the ``straightest path"   i.e., the curves for which the vector field of velocities is parallel along the curve, with respect to a natural connection. 
While in  Riemannian geometry these two approaches lead to the same set of trajectories, in proper sub-Riemannian geometry (i.e., when $D\neq TM$) they lead to different sets of trajectories (see \cite{Alekseevsky2020} for details), and in general, for the second approach, the natural connection only exists under additional (and rather restrictive) assumptions of constancy of sub-Riemannian symbol (\cite{Morimoto2008}).

In the present paper, we consider the geodesic defined by the variational approach.
A horizontal curve $\gamma:[a, b]\rightarrow  M$ is an absolutely continuous curve tangent to $D$, i.e., $\gamma'(t) \in D \left(\gamma(t)\right)$. In the sequel, the manifold $M$ is assumed to be connected. By the Rashevskii-Chow theorem the assumption that  $D$ is bracket-generating guarantees that the space of horizontal curves connecting two given points $q_0$ and $q_1$ is not empty.  The following energy-minimizing problem:
\begin{equation}
\label{energy}
\begin{split}
~& E(\gamma) = \displaystyle\int_a^b g\left(\gamma'(t),\gamma'(t)\right)\, dt\rightarrow \textnormal{min},\\
~&\gamma'(t)\in D \left(\gamma(t)\right) \quad \textnormal{ a.e. } t,\\
~& \gamma(a)=q_0, \quad \gamma(b)=q_1.
\end{split}
\end{equation}
can be solved using the Pontryagin Maximum Principle (\cite{AgrSach, Pont}) in optimal control theory that defines special curves in the cotangent bundle $T^*M$, called the \emph{Pontryagin extremals}, so that a minimizer of the optimal control problem \eqref{energy} is a projection from $T^*M$ to $M$ of some \emph{Pontryagin extremal} (for more explicit description of Pontryagin extremals see the beginning of section \eqref{orb_sec} below). 
\begin{Def}
\label{geod}
The \emph{(variational) sub-Riemannian geodesics} are projections of the Pontryagin extremals of the optimal control problem \eqref{energy}.
\end{Def}

Note that in the Riemannian case geodesics given by Definition \ref{geod}  coincides with the usual Riemannian geodesics.
We thus extend the definitions of projective and affine equivalences of Riemannian metrics to the general sub-Riemannian case  in the following way.

\begin{Def}
Let $M$ be a manifold and $D$ be a bracket-generating distribution on $M$. Two sub-Riemannian metrics  $g_1$ and $g_2$ on  $(M, D)$ are called \emph{projectively equivalent} at $q_0\in M$ if they have the same geodesics, up to a reparameterization, in a neighborhood of $q_0$. They are called \emph{affinely equivalent} at $q_0$ if they have the same geodesics, up to affine reparameterization, in a neighborhood of $q_0$. 
\end{Def}

Again, we will write   $g_1\stackrel{p}{\sim} g_2$ and $g_1\stackrel{a}{\sim} g_2$ in the case of projective and affine equivalence, respectively.
By complete analogy with the Riemannian case,  for a sub-Riemannian metric $g$ on $(M, D)$  and a positive constant $c$  the metrics $c g$ and $g$ are affinely equivalent. The metric $c g$ will be said a \emph{constantly proportional metric} to the metric $g$. 

\begin{Def}
A sub-Riemannian metric $g$ on $(M, D)$ is called \emph{affinely rigid} if the sub-Riemannian metrics constantly proportional to it are the only 
sub-Riemannian metrics on $(M, D)$ that are affinely equivalent to $g$. 
\end{Def}
As in the Riemannian case, examples of affinely nonrigid sub-Riemannian structures can be constructed with the help of the appropriate notion of product structure. For this we first have to define distributions admitting  product structure as follows: 

\begin{Def}
A distribution $D$ on a manifold $M$ \emph{admits a product structure} if there exist two manifold $M_1$ and $M_2$ of positive dimension  endowed with  two distributions $D_1$ and $D_2$ of positive rank  (on $M_1$ and $M_2$, respectively)  such that the following two conditions holds:
\begin{enumerate}
\item $M=M_1\times M_2$;
\item If $\pi_i: M\rightarrow M_i$, $i=1,2$, are the canonical projections and 
$\pi_i^*D_i$ denotes the pullback of the distribution $D_i$ from $M_i$ to $M$, i.e.,
$$\pi_i^* D_i (q)=\{v\in T_q M: d \pi_i(q) v\in D_i\bigl(\pi_i(q)\bigr)\},$$ then 
\begin{equation} 
\label{proddist}
D(q)=\pi_1^*D_1(q)\cap \pi_2^*D_2(q).
\end{equation}
\end{enumerate}
In this case, we will write that $(M, D)=(M_1,D_1)\times (M_2, D_2)$. 
\end{Def} 
\begin{Def}
A sub-Riemannian structure 
$(M,D,g)$ \emph{admits a product structure} if there exist (nonempty) sub-Riemannian structures $(M_1, D_1, g_1)$ and $(M_2, D_2, g_2)$ such that $(M,D)=(M_1,D_1)\times (M_2,D_2)$ and 
 if $\pi_i: M\mapsto M_1$ are the canonical projections then identity \eqref{prodmetric} holds.
In this case we will write that $(M,D,g)=(M_1,D_1,g_1)\times (M_2,D_2, g_2)$.
\end{Def}

It is easy to see that if $(M,D,g)=(M_1,D_1,g_1)\times (M_2,D_2, g_2)$ then this sub-Riemannian metric is affinely equivalent to 
\begin{equation}
\label{nonconstprop}
(M_1,D_1,c_1g_1)\times (M_2,D_2, c_2g_2)
\end{equation}
for every two positive constants $c_1$ and $c_2$, but the latter metric is not constantly proportional to $(M, D, g)$ , if $c_1\neq c_2$, i.e., a \emph{sub-Riemannian  metric  admitting product structure} is not affinely rigid. The main question is whether or not the converse of this statement, at least in a local setting, i.e., the analog of the Eisenhart theorem (Theorem \ref{eisenthm}) holds.

\begin{Conj} [\cite{jmz2019}] 
\label{conjecture}
If a sub-Riemannian metric $g$  is not affinely rigid near a point $q_0$, i.e., admits a locally affinely equivalent non-constantly proportional  sub-Riemannian metric in a neighborhood of a point $q_0$,  
	then the metric $g$ is the direct product of two sub-Riemannian metrics in a neighborhood of $q_0$.
\end{Conj}

In this paper, we prove this conjecture for sub-Riemannian metrics on a class of step 2 distributions, see Theorem \ref{mainthm}.

\section{The role of Tanaka symbol/nilpotent approximation and the main result}

Conjecture \ref{conjecture} is still widely open. In the present paper, we prove it for sub-Riemannian metrics on a particular, but still rather large class of distributions (see Theorem \ref{mainthm} below).  To formulate our main result (Theorem \ref{mainthm}) we need to introduce some terminology.

\subsection{Direct product structure on the level of Tanaka symbol/nilpotent approximation}
In \cite{jmz2019} we proved,  among other things, a weaker product structure result for affinely nonrigid sub-Riemannian structures, in which the product structure necessarily occurs on the level of Tanaka symbol/ nilpotent approximation of the the sub-Riemannian structure. 

To define the Tanaka symbol of the distribution $D$ at a point $q$ we need another assumption on $D$ near $q$, called \emph{equiregularity}.
 A distribution $D$ is called \emph{equiregular} at a point $q$ if there is a neighborhood $U$ of $q$ in $M$ such that for every $j>0$ the dimensions of subspaces $D^j(y)$ are constant for all $y\in U$, where $D^j$  is in \eqref{Dpower1} (equivalently, as in \eqref{Dpower2}). 
 Note that a bracket-generating distribution is equiregular at a generic point.

From now on we assume that $D$ is an equiregular bracket-generating distribution with the degree of nonholonomy $\mu$. 
Set 
\begin{equation}
\label{Tanaka_comp}
\mathfrak{m}_{-1}(q):=D(q), \quad  \mathfrak{m}_{-j}(q) : =D^{j}(q)/D^{j-1}(q), \,\, \forall j>1 
\end{equation}
and  consider the graded space
\begin{equation}
\label{symbdef}
\mathfrak{m}(q)=
\bigoplus_{j=-\mu}^{-1}\mathfrak{m}_j(q),
\end{equation}
associated with the filtration \eqref{filtr}.

The space $\mathfrak{m}(q)$  is endowed with the natural structure of a graded Lie algebra, i.e., with the natural Lie product $[\cdot, \cdot]$ such that
\begin{equation}
 \label{gradcond}
[\mathfrak{m}_{i_1}(q), \mathfrak{m} _{i_2}(q)]\subset \mathfrak{m}_{i_1+i_2}
\end{equation}
defined as follows: 

Let $\mathfrak{p}_j:D^j(q)\mapsto \mathfrak m_{-j}(q)$ be the canonical projection to a factor space. Take $Y_1\in\mathfrak m_{-i_1}(q)$ and $Y_2\in \mathfrak m_{-i_2}(q)$.
To define the Lie bracket $[Y_1,Y_2]$ take a local section $\widetilde Y_1$ of the distribution $D^{i_1}$ and
a local section $\widetilde Y_2$ of the distribution $D^{i_2}$ such that
\begin{equation}
\label{projbracket}
\mathfrak p_{i_1}\bigl(\widetilde Y_1(q)\bigr) =Y_1,\quad \mathfrak p_{i_2}\bigl(\widetilde Y_2(q)\bigr)=Y_2.
\end{equation}
It is clear from definitions of the spaces $D^j$ that $[\widetilde Y_1,\widetilde Y_2]\in D^{i_1+i_2}$. Then set
\begin{equation}
\label{Liebrackets}
[Y_1,Y_2]:=\mathfrak p_{i_1+i_2}\bigl([\widetilde Y_1,\widetilde Y_2](q)\bigr).
\end{equation}
It can be  shown (\cite{tan1, zeltan}) that the right hand-side  of \eqref{Liebrackets} does not depend on the choice of sections $\widetilde Y_1$ and
$\widetilde Y_2$. By constructions, it is also clear that \eqref{gradcond} holds.

\begin{Def}
The graded Lie algebra $\mathfrak{m}(q)$ from \eqref{symbdef} is called the \emph{symbol} of the distribution $D$ at the point $q$.
\end{Def}

By constructions, it is clear that the Lie algebra $\mathfrak{m}(q)$ is nilpotent.
The Tanaka symbol is the infinitesimal version of the so-called \emph{nilpotent approximation} of the distribution $D$ at $q$, which can be defined as the left-invariant distribution $\widehat D$ on the simply connected Lie group with the Lie algebra $\mathfrak{m} (q)$ and the identity $e$, such that $\widehat D(e)=\mathfrak{m}_{-1}(q)$. 

Further, since $D$ is bracket-generating, its Tanaka symbol $\mathfrak{m}(q)$ at any point is generated by the component $\mathfrak{m}_{-1}(q)$.

\begin{Def}
\label{fund}
A (nilpotent) $\mathbb Z_-$- graded Lie algebra 
\begin{equation}
\label{fundgrad}
\mathfrak{m}=\bigoplus_{j=-\mu}^{-1}\mathfrak{m}_j
\end{equation}
is called a \emph{fundamental} graded Lie algebra (here $\mathbb Z_-$ denotes the set of all negative integers), if it is generated by $\mathfrak m_{-1}$.
\end{Def}
The following notion will  be crucial in the sequel:

\begin{Def}
A fundamental graded Lie algebra $\mathfrak{m}$ is called \emph{decomposable} if it can be represented as a direct sum of two nonzero fundamental graded Lie algebras $\mathfrak{m}^1$ and $\mathfrak{m}^2$ and it is called \emph{indecomposable} otherwise. Here the $j$th component of $\mathfrak{m}$ is the direct sum of the $j$th components of $\mathfrak{m}^1$ and $\mathfrak{m}^2$.   
\end{Def}
Obviously, if a distribution $D$ admits product structure then its Tanaka symbol at any point is decomposable.

\begin{example} [contact and even contact distributions]
Assume that $D$ is a corank 1 distributions, $\dim D(q)=\dim M-1$ , and assume $\alpha$ is its defining $1$-form, i.e.,   $D=\mathrm{ker} \,\alpha$ . 
\begin{itemize}
\item Recall that the distribution $D$ is called \emph{contact} if $\dim M$ is odd and the form $d\alpha|_D$ is non-degenerate. In this case the Tanaka symbol at a point $q$ is isomorphic to the Heisenberg algebra $\mathfrak{m}_{-1}(q)\oplus\mathfrak{m}_{-2}(q)$ of  dimension equal to $\dim M$, where $\mathfrak{m}_{-2}(q)$ is the (one-dimensional) center and the brackets on $\mathfrak{m}_{-1}(q)$ ($\cong D(q)$) are
 given by $[X, Y]:=d \alpha(X, Y) Z$, where $Z$ is the generator of $\mathfrak{m}_{-1}$ so that $\alpha(Z)=1$. Note that the Heisenberg algebra is indecomposable as the fundamental graded Lie algebra. Otherwise, since 
 $\dim \mathfrak m_{-2}(q)=1$ one of the components in the nontrivial decomposition of $\mathfrak m(q)$ will be commutative and belong to $\mathfrak m_{-1}(q)$ and hence to the kernel of $d\alpha|_D$, which contradicts the condition of contactness.
\item Recall that the distribution $D$ is called \emph{quasi-contact} (in some literature \emph{even contact}) if $\dim M$ is even and the form $d\alpha|_D$ has a one-dimensional kernel (i.e., of the minimal possible dimension for a skew-symmetric form on an odd-dimensional vector space). In this case by the arguments similar to the previous item the Tanaka symbol is the direct some of the Heisenberg algebra (of dimension $\dim M-1$) and $\mathbb R$ (the kernel of $d\alpha|_D$), i.e., the Tanaka symbol is decomposable.
\end{itemize}
\end{example}
\begin{rk}
\label{unique_decomp}
\emph{It is easy to show that the decomposition of a fundamental graded $\mathfrak{m}$ Lie algebra into indecomposable fundamental Lie algebras is unique modulo the center of $\mathfrak{m}$ and a permutation of components.} 
\end{rk}
The following theorem is a consequence of the results proved in \cite{jmz2019}  and it is a weak version of  Conjecture \ref{conjecture}:

\begin{theorem}
\cite[a  consequence of Theorem 7.1, Proposition 4.7, and Corollary 4.9 there]{jmz2019}
\label{Tanaka_decomp_thm}
If a sub-Riemannian metric on an equiregular distribution $D$ is not affinely rigid near a point $q_0$ then its Tanaka symbol at $q_0$ is decomposable. 
\end{theorem}

In other words, \emph{the problem of affine equivalence is nontrivial only on the distributions with decomposable Tanaka symbols} (at points where the distribution is equiregular).

\subsection{Ad-surjective Tanaka symbols and the main result}
Now we are almost ready to formulate the main result of the paper. We restrict ourselves here to step $2$ distributions, i.e., when $D^2=TM$. Such distributions are automatically equiregular (at any point). Then it is clear that the components in the decomposition of the Tanaka symbols of such distribution are of step not greater than $2$ (i.e., with $\mu\leq 2$ in \eqref{fundgrad}). So, they are either of step $2$ or commutative.   

\begin{Def}
\label{ad_surj_def}
We say that a step $2$ fundamental graded Lie algebra $\mathfrak{m}=\mathfrak{m}_{-1}\oplus \mathfrak{m}_{-2}$ is \emph{$\mathrm{ad}$-surjective} if there exists $X\in \mathfrak{m}_{-1}$ such that the map $\mathrm{ad} X: \mathfrak{m}_{-1}\rightarrow \mathfrak{m}_{-2}$, 
$$Y\mapsto [X, Y], \quad Y\in\mathfrak{m}_{-1},$$
is surjective. An element $X\in \mathfrak{m}_{-1}$ for which $\mathrm{ad} X$ is surjective is called an \emph{$\mathrm{ad}$-generating element} of the algebra $\mathfrak m$.
\end{Def}
\begin{rk} 
\label{sum_ad_surj}
\emph{Note that the direct sum $\mathfrak {m}^1\oplus\mathfrak {m}^2$ of two $\mathrm{ad}$-surjective Lie algebras $\mathfrak {m}^i=\mathfrak {m}_{-1}^i\oplus \mathfrak {m}_{-2}^i$, $i=1,2$, is $\mathrm{ad}$-surjective. Indeed, if $X_i\in \mathfrak {m}^i_{-1}$, $i=1,2$, are such that  $\mathrm{ad} \, X_i: \mathfrak {m}^i_{-1}\rightarrow \mathfrak {m}^i_{-2}$ is surjective, then $$\mathrm{ad} (X_1+X_2): \mathfrak {m}^1_{-1}\oplus\mathfrak{m}^2_{-1} \rightarrow \mathfrak{m}^1_{-2}\oplus\mathfrak{m}^2_{-2}$$ is surjective as well. And vice versa, if a step $2$ fundamental Lie algebra $\mathfrak m$ is $\mathrm {ad}$-surjective, then any component of its decomposition into fundamental graded Lie algebra is $\mathrm {ad}$-surjective: the projection of an $\mathrm{ad}$-generating element of $\mathfrak {m}$ to any component is $\mathrm{ad}$-generating element of this component}.
\end{rk}
\begin{rk}
\emph{Nilpotent Lie algebras of H-type, introduced by A. Kaplan  \cite{kaplan} in 1980, and extensively studied since then are $\mathrm{ad}$-surjective, because, among other properties of Lie algebras of H-type, it is required that every element of $\mathfrak m_{-1}$ has to be $\mathrm{ad}$-generating.}        
\end{rk}

The following proposition will be proved in the Appendix \ref{appendix A}: 

\begin{Prop}
\label{codim3}
Any step $2$ fundamental graded Lie algebra $\mathfrak{m}=\mathfrak{m}_{-1}\oplus \mathfrak{m}_{-2}$ such that the following three condition holds
\begin{enumerate}

\item $\dim\, \mathfrak{m}_{-2}\leq 3$;

\item $\dim\, \mathfrak{m}_{-2}<\dim \mathfrak{m}_{-1}$; \item
the intersection of $\mathfrak{m}_{-1}$ with the center of $\mathfrak{m}$ is trivial
\end{enumerate}
is  $\mathrm{ad}$-surjective. \end{Prop}
Note that if $\dim\, \mathfrak{m}_{-2}\leq 2$ then the item (2) of the previous proposition holds automatically. Besides, obviously item (2) is a necessary condition for $\mathrm{ad}-$ surjectivity.

\begin{Coro}
\label{codim3cor}
The only non-$\mathrm{ad}$-surjective step $2$ fundamental graded Lie algebra with $\mathfrak{m}_{-2}\leq 3$ is the truncated step $2$ free Lie algebra with $3$ generators.  
\end{Coro}
Note that Proposition \ref{codim3} does not hold if one drops item (1), see Appendix \ref{appendix A}, Example \ref{counter_example}, for a family of counter-examples with $\dim \mathfrak{m}_{-2}=4$ and $\dim \mathfrak{m}_{-1}=5$. These counter-examples are semi-direct sums of the truncated step $2$ free Lie algebras with three generators and the $3$-dimensional Heisenberg algebra.

Nevertheless,  following \cite[Section 8]{jmz2019}, given integers $m>0$ and $d\geq 0$, if we denote by $\mathrm{GLNA} (m, m+d)$  the set of all  fundamental graded nilpotent Lie algebras $\mathfrak m$ of step not greater than $2$ satisfying 
\begin{equation}
\label{dim-1-2}
\dim\, \mathfrak m_{-1}=m,\quad   \dim\,\mathfrak m_{-2}=d,
\end{equation}
we have the following genericity results:
\begin{Lemma}

\label{gen_lemma}
If  $m>d$, the subset of all  
$\mathrm{ad}$-surjective graded nilpotent Lie algebras belonging to $\mathrm{GLNA} (m, m+d)$ 
is generic in $\mathrm{GLNA} (m, m+d)$.

\end{Lemma}
\begin{proof}
Indeed,  for such Lie algebras, the Lie algebra structure is encoded   by the Levi operator  $\mathcal L_q\in {\rm Hom}\bigl( \bigwedge^2 \mathfrak m_{-1}, \mathfrak m_{-2}\bigr)$ which  is defined as follows:
\begin{equation}
\mathcal L(X,Y)=[X,Y], \quad \forall X,Y\in \mathfrak m_{-1},
\end{equation}
and the fundamentality assumption implies that $\mathcal L$ is surjective.
Equivalently, one can consider the  dual operator  $\mathcal L^*\in {\rm Hom}\bigl((\mathfrak m_{-2})^*,  \bigwedge^2 (\mathfrak m_{-1})^*)$,
\begin{equation}
\label{A*}
\mathcal L^*(p)(X,Y)=p([X,Y])\quad X,Y\in \mathfrak m_{-1}, \,\,p\in  (\mathfrak m_{-2})^*.
\end{equation}
Here we use the natural identification 
$\Bigl(\bigwedge^2 \mathfrak m_{-1}\Bigr)^*\cong\bigwedge^2 (\mathfrak m_{-1})^*$, which in turn is naturally identified with the space of skew-symmetric bilinear forms on $\mathfrak m_{-1}.$ Note that, again from the surjectivity of $\mathcal L$, its dual $\mathcal L^*$  is injective and is described by its image, which is a $d$-dimensional space. So, 
the space of  
all fundamental graded nilpotent Lie algebras of step not greater than $2$  satisfying \eqref{dim-1-2} is isomorphic to the Grassmannian of $d$-dimensional subspaces in the space of skew-symmetric forms of an  $m$-dimensional vector space, modulo the natural action of the General Linear group on this space. In particular, the latter Grassmannian is a connected algebraic variety and the subset of $\mathrm{ad}$-surjective graded nilpotent Lie algebras of step not greater than $2$  satisfying \eqref{dim-1-2} with $m>d$ corresponds to a non-empty Zariski open subset of it, therefore it is generic.
\end{proof}

\begin{rk} \emph{By \cite[Proposition 8.1]{jmz2019} the subset of indecomposable graded nilpotent Lie algebras in $\mathrm{GLNA}(m, m+d)$ is generic $\mathrm{GLNA}(m, m+d)$ for all pairs $(m, d)$ with the exception of the following three cases:
\begin{enumerate}
    \item $d=0$, $m>1$ (Riemannian case of dimension greater than $1$;
    \item $d=1$, $m>1$ and  odd (the quasi-contact case);
    \item $d=2$, $m=4$.
\end{enumerate}
Moreover, in cases (1) and (2) all graded Lie algebras in $\mathrm{GLNA}(m, m+d)$ are decomposable, while in case (3) the set of decomposable fundamental symbols is nonempty open and corresponds to symbols for which the set of solutions of the equation $\mathcal L^*(p)\wedge \mathcal L^*(p)=0$ considered as the equation with respect to $p\in (\mathfrak m_{-2})^*$, where   $L^*(p)$ is as in \eqref{A*}, consists of two distinct (real) lines.}
\end{rk}

The main result of the present paper is the following
\begin{theorem}
\label{mainthm}
Assume that $D$ is a step 2 distribution such that its Tanaka symbol is $\mathrm {ad}$-surjective. 
 If a sub-Riemannian metric  $(M, D, g_1)$ is not affinely rigid near a point $q_0$, then it admits a product structure in a neighborhood of $q_0$.  
\end{theorem}

\begin{rk}
\label{ad_surj_for any_rem}
\emph{First note that by Theorem \ref{Tanaka_decomp_thm} under the hypothesis of the previous theorem the Tanaka symbol of $D$ must be decomposable and by the second sentence of Remark \ref{sum_ad_surj} all components of this decomposition are $\mathrm{ad}$-surjective. 
Second, by Remark \ref{unique_decomp}, if such decomposition consist of indecomposable components only,  the number of these components is independent of the decomposition. Let us denote this number by $\hat k$. 
Then the sub-Riemannian metric in Theorem \ref{mainthm} is a product of at least two and at most $\hat k$ sub-Riemannian structures each of which is affinely rigid (in the neighborhood of the projection of $q_0$ to the corresponding manifold).
}
\end{rk}
The rest of the paper is devoted to the proof of Theorem \ref{mainthm}. This theorem confirms the Conjecture \ref{conjecture} for sub-Riemannian metrics on step 2 distributions with $\mathrm{ad}$-surjective Tanaka symbol. 
As a direct consequence of Theorem \ref{mainthm} and Corollary \ref{codim3cor}  we get the following

\begin{Coro}
Assume that $D$ is a step 2 distribution such that its Tanaka symbol is decomposed  into $\hat k\geq 2$ 
nonzero indecomposable fundamental graded Lie algebras
with degree $-2$ components of dimension not greater than $3$ and such that among them there is no truncated step $2$ free Lie algebra with $3$ generators.  If a sub-Riemannian metric  $(M, D, g_1)$ is not affinely rigid near a point $q_0$ then it admits a product of 
of at least two and at most $\hat k$ sub-Riemannian structures each of which is affinely rigid (in the neighborhood of the projection of $q_0$ to the corresponding manifold).
\end{Coro}

The assumption of $\mathrm{ad}$-surjectivity  of the  Tanaka symbol is crucial for our proof of Theorem \ref{mainthm} because we strongly use a natural quasi-normal form for  
$\mathrm{ad}$-surjective Lie algebras, see \eqref{quasi_complete}. We hope that analogous quasi-normal forms can be found for more general graded nilpotent Lie algebras so that Conjecture \ref{conjecture} can be proved similarly for a more general class of sub-Riemannian metrics.

\section{Orbital equivalence and  fundamental algebraic system}
\label{orb_sec}

In general, there are two	types of Pontryagin  extremals for optimal control problems, \emph{normal} and \emph{abnormal} (\cite{ABB, AgrSach}): for the former, the Lagrange multiplier near the functional is not zero and for the latter, it is zero.  In particular, abnormal extremals, as unparametrized curves,  depend on the distribution $D$  only and not on a metric $g$ on it. This indicates that only normal extremals are essential for the considered problems of affine/projective equivalence (see Proposition  \ref{th:orb.to.proj}
for the precise formulation). Therefore, we give an explicit description of normal extremals only.  They  are the  integral curves of the Hamiltonian vector field  $\vec h$ on $T^*M$ corresponding to the Hamiltonian 
\begin{equation}
\label{Hamiltonian}
h(p,q)=||\, p|_{D(q)}||^2, \quad q\in M, p\in T^*_qM,
\end{equation}
and lying on a nonzero level set of $h$. Here $||\, p|_{D(q)}||$ denotes the operator norm of the functional  $p|_{D(q)}$, i.e., 

$$||\, p|_{D(q)}||=\max \{p(v): v\in D(q), g(q)(v,v)=1\}.$$
The Hamiltonian $h$ defined by \eqref{Hamiltonian} is 
called the \emph{Hamiltonian, associated with the metric $g$} or shortly the \emph{sub-Riemannian Hamiltonian}.		 
	

In \cite{jmz2019}, following \cite{Z}, the problems of projectively and affine equivalence of sub-Riemannian metric were reduced to the study of the orbital equivalence of the corresponding sub-Riemannian Hamiltonian systems for normal Pontryagin extremals of the energy minimizing problem \eqref{energy}, which in turn is reduced to the study of solvability of a special linear algebraic system with coefficients being polynomial in the fibers, called the \emph{fundamental algebraic system} (\cite [Proposition 3.10]{jmz2019}). In this section we summarize all information from \cite{jmz2019} we need for the proof of Theorem \ref{mainthm}. 



As before, fix a connected manifold $M$ and a bracket-generating equiregular distribution $D$ on $M$, and consider two sub-Riemannian metrics $g_1$ and $g_2$ on $(M, D)$. We denote by $h_1$ and $h_2$ the respective sub-Riemannian Hamiltonians of $g_1$ and $g_2$, as defined in \eqref{Hamiltonian}. Let the annihilator $D^\perp$ of $D$ in $T^*M$ be defined as follows: 
\begin{equation}
\label{perp}
D^\perp=\{(p, q)\in T^*M: p|_{D(q)}=0\}
\end{equation}
It  coincides with the zero level set of the sub-Riemannian Hamiltonian $h$ from \eqref{Hamiltonian}.
\begin{Def}
We say that $\vec h_1$ and  $\vec h_2$ are \emph{orbitally diffeomorphic}
on an open subset $V_1$ of $T^*M\backslash D^\perp$  if there exists an open subset $V_2$ of  $T^*M\backslash D^\perp$ and a diffeomorphism $\Phi: V_1 \rightarrow V_2$ such that $\Phi$ is fiber-preserving, i.e.,  $\pi (\Phi(\lambda)) = \pi (\lambda)$, and $\Phi$ sends the integral curves of $\vec h_1$ to the reparameterized integral curves of $\vec h_2$, i.e., there exists a smooth function $s=s(\lambda,t)$ with $s(\lambda,0)=0$ such that $\Phi\bigl(e^{t\vec h_1}\lambda\bigr) =e^{s\vec h_2} \bigl(\Phi(\lambda)\bigr)$ for all $\lambda \in V_1$ and $t\in \mathbb R$ for which $e^{t\vec h_1}\lambda$ is well defined. Equivalently, there exists a smooth function $\aa(\lambda)$ such that
\begin{equation}
\label{orbeq}
d\Phi\, \vec{h}_1(\lambda) = \aa(\lambda) \vec{h}_2(\Phi(\lambda)).
\end{equation}
The map $\Phi$ is called an \emph{orbital diffeomorphism} between the extremal flows of $g_1$ and  $g_2$.
\end{Def}
The reduction of projective (respectively,  affine) equivalence of sub-Riemannian metrics  to the orbital ( respectively, a special form of orbital) equivalence of the corresponding sub-Riemannian Hamiltonian systems is given by the following:

\begin{Prop}\cite[a combination of Proposition 3.4  and Theorem 2.10 there]{jmz2019}
\label{th:orb.to.proj}
Assume that the sub-Riemannian metrics $g_1$ and $g_2$ are projectively equivalent in a neighborhood  $U\subset M$ and let $\pi: T^*M\rightarrow M$ be the canonical projection.  Then, for generic point $\lambda_1 \in \pi^{-1} (U)\backslash D^\perp$,  $\vec{h}_1$ and $\vec{h}_2$ are orbitally diffeomorphic on a neighborhood $V_1$ of $\lambda_1$ in $T^*M$.
Moreover, if $g_1$ and $g_2$ are affinely equivalent in a neighborhood  $U\subset M$, then the function $\aa(\lambda)$ in~\eqref{orbeq} satisfies $\vec{h}_1(\aa)=0$. \end{Prop}

Further, the differential equation \eqref{orbeq} can be written (\cite[Lemma 3.8]{jmz2019}) and transformed to the algebraic system via a sequence of prolongations  (\cite[Proposition 3.9]{jmz2019}) in a special moving frame adapted to the sub-Riemannian structures $g_1$ and $g_2$. For this, first, we need the following

\begin{Def}
The \emph{transition operator} at a point $q \in M$ of the pair of metrics $(g_1, g_2)$ is the linear operator $S_q :D (q) \rightarrow D(q)$ such that 
\begin{equation}
\label{transition_op_eq}
g_2(q)(v_1, v_2)=g_1(q)(S_q v_1, v_2),\quad  \forall v_1, v_2 \in D(q).
\end{equation}
\end{Def}
Obviously $S_{q}$ is a positive $g_1$-self-adjoint operator and its eigenvalues $\alpha^2_1(q)$, \dots, $\alpha^2_m(q)$ are positive real numbers (we choose $\alpha_1(q), \dots, \alpha_m(q)$ as positive numbers as well). The field $S$ of transition operators is a $(1,1)$-tensor field that will be called the \emph{transition tensor}. 

The important simplification in the case of  the affine equivalence compared to the projective equivalence is given in the following

\begin{Prop} \cite[Propostion 4.7]{jmz2019}
\label{le:constant_alphai}
If two sub-Riemannian metrics $g_1, g_2$ on $(M, D)$ are affinely equivalent on an open connected subset $U \subset M$, then all the eigenvalues $\alpha_1^2, \dots,\alpha_m^2$ of the transition operator are constant.
\end{Prop}
This proposition implies that the number of the distinct eigenvalues $k(q)$ of the transition operators $S_q$ is independent of $q\in U$, i.e., $k(q)\equiv k$ on $U$ for some positive integer $k$. Also, there is $k$ distributions $D_i$ such that  
\begin{equation}
\label{eigendecomp}
D(q)=\displaystyle{\bigoplus_{i=1}^k D_i(q)}
\end{equation}
is the eigenspace decomposition of $D(q)$ with respect to the eigenspaces of the operator $S_q$. 
Now let
\begin{equation}
\label{weakdecompsymb_1} 
\mathfrak{m}_{-1}^i(q)=D_i(q), \quad  \mathfrak{m}_{-j}^i(q)=(D_i)^j(q)/\left((D_i)^j(q)\cap D^{j-1}(q)\right), \,\,\forall j>1. \footnote{ Since $(D_i)^j\subset D^j$, the space $\mathfrak{m}_{-j}^i$ is a subspace of $\mathfrak{m}_{-j}$.}
\end{equation}
Set 
\begin{equation}
\label{weakdecompsymb_2}
\mathfrak{m}^i(q)=\bigoplus_{j=1}^\mu  \mathfrak{m}_{-j}^i(q).
\end{equation}
By construction $\mathfrak{m}^i$, $i=1, \ldots k$, are fundamental graded Lie algebras.
\begin{rk}
\emph{Note that in general $\mathfrak m^i(q)$ is not equal/isomorphic to the Tanaka symbol of the distribution $D_i$ at $q$, as when defining the components  $
\mathfrak m_{-j}^i(q)$ with $j>1$ we also make the quotient by the powers of $D$. In fact, the proof that $\mathfrak m^i(q)$ is isomorphic to the Tanaka symbol of the distribution $D_i$ under the assumption of affine nonrigidity is one of the main steps in the proof of Theorem \ref{mainthm}.}
\end{rk}

\begin{Prop}\cite[Theorem 6.2 and Theorem 7.1]{jmz2019} 
\label{Tanaka_decomp_theorem}
If sub-Riemannian metrics $g_1, g_2$ are affinely  equivalent and not constantly proportional to each other in a connected open set $U$, then for every $q\in U$ the Tanaka symbol 
$\mathfrak m(q)$ of the distribution $D$ at $q$ is the direct sum  of the fundamental graded Lie algebras $\mathfrak{m}^i(q)$, $i=1, \ldots k$ defined by \eqref{weakdecompsymb_1} and \eqref{weakdecompsymb_2}, i.e.,
\begin{equation}
 \label{Tanaka_decomposition}
 \mathfrak m(q)=\bigoplus_{i=1}^k \mathfrak m^i(q),
\end{equation}
as the direct sum of Lie algebras.
\end{Prop}
\medskip

Further, in a neighborhood $U_1$ of any point $q_0\in U$ we can choose a $g_1$-orthonormal local  frame $X_1, \dots, X_m$ of $D$ whose values at any $q \in U_1$ diagonalizes $S_q$, i.e., $X_i(q)$ is an eigenvectors of $S_q$ associated with the eigenvalues  $\alpha^2_i(q)$, $i=1, \ldots m$. Note that $\frac{1}{\alpha_1}X_1, \dots, \frac{1}{\alpha_m}X_m$ form a $g_2$-orthonormal frame of $D$.  We then complete $X_1, \dots, X_m$ into a frame $\{X_1, \dots, X_n\}$ of $TM$ adapted to the distribution $D$ near $q_0$, i..e such that for every positive integer $j$ this frame contain a local frame of $D^j$. We call such a set of vector fields $\{X_1, \dots, X_n\}$  a \emph{(local) frame adapted to the (ordered) pair of metrics} $(g_1, g_2)$.  
The \emph{structure coefficients} of the frame $\{X_1, \dots, X_n\}$ are the real-valued functions $c^k_{ij}$, $i,j,k \in \{1,\dots,n\}$ defined near $q$ by
\begin{equation}
\label{struct_function}
[X_i,X_j] = \sum_{k=1}^n c^k_{ij} X_k.
\end{equation}
Let $u = (u_1, \dots, u_n)$ be the coordinates on the fibers $T_q^*M$  induced by this frame, i.e., 
\begin{equation}
\label{ui}
u_i(q,p) =  p \left(X_i(q)\right).
\end{equation}
These coordinates in turn induce a basis $\partial_{u_1}, \dots, \partial_{u_n}$ of $T_{\lambda}(T^*_qM)$ for any $\lambda \in \pi^{-1}(q)$. For $i = 1, \dots, n$, we define the lift $Y_i$ of $X_i$ as the (local) vector field on $T^*M$ such that $\pi_* Y_i = X_i$ and $du_j (Y_i) = 0 \ \ \forall 1 \leq j \leq n$. The family of vector fields $\{ Y_1, \dots, Y_n, \partial_{u_1}, \dots, \partial_{u_n}\}$ obtained in this way is called a \emph{frame of $T(T^*M)$ adapted at} $q_0$. By a standard calculation, we obtain the expression for the sub-Riemannian Hamiltonian $h_1$ of the metric $g_1$ and the corresponding  Hamiltonian vector field $\vec h_1$:
\begin{align}
~& h_1 = \frac{1}{2}\sum_{i=1}^{m} u_i^2 \label{Ham_h1}\\
~& \vec{h}_1 =\sum_{i=1}^{m} u_i\vec u_i = \sum_{i=1}^{m} u_i Y_i + \sum_{i=1}^{m}\sum_{j,k=1}^{n} c^{k}_{ij} u_i u_k \partial_{u_j}.\label{eq:vech}
\end{align}

Indeed,  to prove \eqref{eq:vech},  recall that if for a vector field $Z$ in $M$, we denote 
$$H_Z(p, q)=p(Z(q)), \quad q, p\in T_q^*M,$$ then for  any two vector fields $Z_1$ and $Z_2$ on $M$ 
we have the following identities 
\begin{equation}
\label{poison1}
\overrightarrow {H_{Z_1}} (H_{Z_2})= d\, H_{Z_2}\bigl(\overrightarrow {H_{Z_1}}\bigr) =H_{[Z_1, Z_2]}.
\end{equation}
From this and \eqref{struct_function} it follows immediately that \begin{equation}
\label{Ham_lift}
  \begin{split}  \vec{u_i} = Y_i + \sum ^{n}_{j=1}\vec{u_i}(u_j)\partial_{u_j}
=  Y_i + \sum ^{n}_{j=1}\sum ^{n}_{k=1}c_{ij}^{k}u_k\partial_{u_j}, 
\end{split}
\end{equation}
which immediately implies \eqref{eq:vech}.

Assume now that $\vec{h}_1$ and $\vec{h}_2$ are orbitally diffeomorphic near $\lambda_0\in H_1\cap \pi^{-1}(q_0)$ and let $\Phi$ be the corresponding orbital diffeomorphism. 
Let us denote by $\Phi_i, \ i= 1, \dots, n,$ the coordinates $u_i$ of $\Phi$ on the fiber, i.e.,\ $u\circ \Phi(\lambda) = (\Phi_1(\lambda), \Phi_2(\lambda), \dots, \Phi_n(\lambda))$. Then first it is easy to see \cite[Lemma 1]{Z} that the function $\alpha$ from \eqref{orbeq} 
 satisfies
 \begin{equation}
 \label{eq:2}
        \alpha=\sqrt{\frac{\sum_{i=1}^{m}\alpha_i^2 u_i^2 }{\sum_{i=1}^m u_i^2 }}
    \end{equation} 
and     
    \begin{equation}
        \Phi_k = \frac{\alpha_k^2 u_k}{\alpha}, \forall 1\leq k\leq m.
    \end{equation} 
In \cite{jmz2019} in order to find the equations for the rest of the components $\Phi_{m+1}, \ldots \Phi_n$ of $\Phi$   we first plugged into \eqref{orbeq} the expression  \eqref{eq:vech} for  $\vec h_1$ and similar expression for $\vec h_2$ and then we made the ``prolongation"  of the resulting differential equation by recursively differentiating it in the direction of $\vec h_1$ and replacing the derivatives of $\Phi_i$'s in the direction of $\vec h_1$   by their expressions from the first step. The resulting system of algebraic equation for  $\Phi_{m+1}, \ldots \Phi_n$  is summarized in the following  
\begin{Prop}\footnote{Since in the present paper we mainly work with the affine equivalence only, for which $\alpha_i$'s are constant and $\vec h_1(\alpha)=0$, the expressions in \eqref{RJ} and \eqref{term.of.b.rec.form} are significantly simpler than in \cite{jmz2019}, where the more general case of the projective equivalence is considered.}\cite[ a combination of Proposition 3.4, Proposition 4.3, Proposition 3.10 applied to the case of  affine equivalence] {jmz2019}
\label{fda}
Assume that the sub-Riemannian metrics $g_1$ and $g_2$ are projectively equivalent in a neighborhood  $U\subset M$ and let $\Phi$ be the corresponding orbital diffeomorphism between the normal extremal flows of $g_1$ and $g_2$ with coordinates $(\Phi_1, \dots, \Phi_n)$. Set $$\widetilde{\Phi}=\alpha(\Phi_{m+1}, \dots, \Phi_n).$$ Let also
\begin{equation}
\label{qjk}
q_{jk} = \sum_{i=1}^m c_{ij}^k u_i
\end{equation}
and
\begin{equation}
\label{RJ}
R_j =
\alpha_j^2 \vec{h}_1(u_j)  
 - \sum_{1 \leq i,k \leq m} c^k_{ij} \alpha_k^2 u_i u_k.
\end{equation}
Then $\widetilde{\Phi}$ satisfies a linear system of equations,
	 \begin{equation}
	 \label{A.phi.B}
	A \widetilde{\Phi} = b,
	 \end{equation}
where $A$ is a matrix with $(n-m)$ columns and an infinite number of rows, and $b$ is a column vector with an infinite number of rows. These infinite matrices can be decomposed in layers of $m$ rows each as
\begin{equation}
\label{eq:A_and_b}
		A = \left(
		 \begin{array}{c}
		A^1 \\
		A^2 \\
		 \vdots \\
		A^{s} \\
		 \vdots \\
		 \end{array}
		 \right)
		 \qquad \hbox{and}
		 \qquad b= \left(
		 \begin{array}{c}
		b^1 \\
		b^2 \\
		 \vdots \\
		b^{s} \\
		 \vdots \\
		 \end{array}
		 \right),
\end{equation}
where the coefficients $a^s_{jk}$ of the $(m \times (n-m))$ matrix $A^s$, $s \in \N$,  are defined by induction as
		 \begin{equation}
		 \label{elem.A}
		 \left\{
		 \begin{aligned}
		& a^1_{j,k} = q_{jk}, \qquad  &1 \leq j \leq m, \ m < k  \leq n,\\
		& a^{s+1}_{j, k} = \vec{h}_1(a^s_{j,k}) + \sum_{l = m+1}^{n}a^s_{j,l} q_{l k}, \qquad  &1 \leq j \leq m, \ m < k  \leq n,
		 \end{aligned}
		 \right.
		 \end{equation}
(note that the columns of $A$ are numbered from $m+1$ to $n$ according to the indices of $\widetilde{\Phi}$) and the coefficients $b^{s}_j$, $1 \leq j \leq m$, of the vector $b^s \in \R^m$ are defined by
\begin{equation}
		 \label{term.of.b.rec.form}
		 \left\{
\begin{aligned}
		& b^{1}_{j} = \ R_j, \\
		& b^{s+1}_{j} = \ \vec{h}_1(b^s_{j}) -\sum_{k = m+1}^{n} a^s_{j,k} \sum_{i=1}^{m} u_i 
		\alpha_i^2 q_{ki} 
		\\
		 \end{aligned}
		 \right. .
\end{equation}
\end{Prop}	

\begin{Def}
\label{FAS_def}
The system \eqref{A.phi.B} with $A$ and $b$ defined recursively by \eqref{elem.A} and \eqref{term.of.b.rec.form} is called the \emph{fundamental algebraic system} for the affine equivalence of the sub-Riemannian metrics  $g_1$ and $g_2$. \footnote{ The column vector $b$ in \eqref{A.phi.B} here corresponds to $\alpha b$ in the notations of the fundamental algebraic system in \cite{jmz2019}. It is more convenient in the context of affine equivalence as all components of $b$ become polynomial in $u_j$'s, i.e., the polynomials on the fibers of $T^*M$.} The subsystem 
 \begin{equation}
	 \label{Ai.phi.Bi}
	A^i \widetilde{\Phi} = b^i
	 \end{equation}
with $A^i$ and $b^i$ as in \eqref{eq:A_and_b} is called the \emph{$i$th layer of the fundamental algebraic system \eqref{A.phi.B}.}
\end{Def}

The matrix $A$ has $n-m$ columns and infinitely many rows and $b$ is the infinite-dimensional column vector. So, the fundamental algebraic system \eqref{A.phi.B} is an over-determined linear system on $(\Phi_{m+1},\dots,\Phi_n)$, and all entries of $A$ and $b$ are polynomials \eqref{elem.A} and \eqref{term.of.b.rec.form} in $u_j$'s. Therefore, all $(n-m+1)\times (n-m+1)$ minors of the augmented matrix $[A|b]$ must be equal to zero. Since all of these minors are polynomials in $u_j$'s,  the coefficient of every monomial of these polynomials is equal to zero. It results in a huge collection of constraints on the structure coefficient $c^k_{ij}$. By discovering and analyzing the monomials with the "simplest" coefficients we were able to prove our main Theorem \ref{mainthm}. This analysis is given, for example, in  Lemmas \ref{lemma step 1}, \ref{lemma step 1 k com}, and \ref{lemma p2}. Quasi-normal forms \eqref{quasi_complete} for $\mathrm{ad}$-surjective Lie algebras were crucial for this analysis.

\section{Proof of Theorem \ref{mainthm}}

Let $(M, D, g_1)$ be a sub-Riemannian metric satisfying the assumptions of Theorem \ref{mainthm}. Assume that $g_2$ is a sub-Riemannian metric that is affinely equivalent and non-constantly proportional to $g_1$ in a neighborhood $U$ of a point $q_0$. Let $k$ be the number of distinct eigenvalues of the transition operators of the pair of metrics $(g_1, g_2)$. As mentioned after Proposition \ref{le:constant_alphai} $k$ is constant on $U$ and by   Remark \ref{ad_surj_for any_rem} we have $2\leq k\leq \hat k$. Consider the  sub-distributions $D_i$, $i\in [1:k]$ defined by \eqref{eigendecomp}, and the algebras $\mathfrak{m}^i(q)$ as in \eqref{weakdecompsymb_1}-\eqref{weakdecompsymb_2}. Hence
\begin{equation}
\label{D12}
D=D_1\oplus D_2 \oplus\ldots \oplus D_ k.
\end{equation}

By Remark \ref{sum_ad_surj} the  algebra $\mathfrak m^i$ for every $i\in [1:k]$ is ad-surjective.

\subsection{Main steps in the proof of Theorem \ref{mainthm}}

Observe that in general  
\begin{equation}
D_i(q)\subset D(q)\cap (D_i)^2(q), \quad  \forall i\in [1:\tilde k].
\label{Tanaka_neq}
\end{equation}
One  can define the canonical projection of quotient spaces 
\begin{equation}
\label{pr_i}
\mathrm{pr}_i:  (D_i)^2(q)/D_i(q)\rightarrow (D_i)^2(q)/\left(D(q)\cap (D_i)^2(q)\right).
\end{equation}
Further, given $X\in D_i(q)$,
we can  define two different operators 
\begin{equation}
\label{two_operators}
\begin{split}
~&(\mathrm{ad}\, X)_{\mathrm{mod}\, D}:D(q)\rightarrow D^2(q)/D(q),\\
~&(\mathrm{ad}\, X)_{\mathrm{mod}\, D_i}:D_i(q)\rightarrow (D_i)^2(q)/D_i(q),
\end{split}
\end{equation}
where in the first case we apply the Lie brackets with $X$ as in the Tanaka symbol of the distribution $D$ at $q$ and in the second case we apply the Lie brackets with $X$ as in the Tanaka symbol of the distribution $D_i$ at $q$.

The main steps in the proof of Theorem \ref{mainthm} are described by the following five propositions together with the final step in subsection \ref{rotation_sec} below:

\begin{Prop}
\label{step1_lemma1}
Assume that $X\in D_i(q)$ is such that the restriction of the map $(\mathrm{ad}\, X)_{\mathrm{mod}\, D}$ to $D_i(q)$ is onto $(D_i)^2(q)/\left(D(q)\cap (D_i)^2(q)\right)$.
Then the  projection $\mathrm{pr}_i$ as in \eqref{pr_i} defines the bijection between the image of the map $(\mathrm{ad}\, X)_{\mathrm{mod}\, D_i}$ and the image of the restriction of the map $(\mathrm{ad}\, X)_{\mathrm{mod}\, D}$ to $D_i(q)$.
\end{Prop}


\begin{Prop}
\label{step1_lemma2}
Assume that  $X\in D_i(q)$ satisfies the assumption of the previous lemma.
Then $(D_i)^2(q)/D_i(q)$ coincides with the image of the map $(\mathrm{ad}\, X)_{\mathrm{mod}\, D_i}$.
\end{Prop}

\begin{Prop}
\label{step1prop}
The following identity holds 
\begin{equation}
 D(q)\cap (D_i)^2(q)=D_i(q) \quad  \forall i\in [1:\tilde k]. \label{Tanaka_coincidence_1}
\end{equation}
\end{Prop}
\begin{Prop}
\label{PropD2D3}
The following identity holds
\begin{equation}
(D_i)^3(q)=(D_i)^2(q) \quad  \forall i\in [1:\tilde k]
\label{Tanaka_coincidence_2}
\end{equation}
and therefore the distribution $D_i^2$ is involutive.
\end{Prop}
\begin{Prop}
 \label{Dij^2_involutive}
 For every two distinct  $r$ and $t$ from $[1:k]$ the distribution $D_r^2+D_t^2$ is involutive.
\end{Prop}


First, let us show that Propositions \ref{step1_lemma1} and \ref{step1_lemma2}  imply  Proposition \ref{step1prop}. Indeed, we have the following chain of inequalities/ equalities:

$$\dim (D_i)^2(q)/D_i(q)\stackrel{\eqref{Tanaka_neq}}{\leq}\dim (D_i)^2(q)/\left(D(q)\cap (D_i)^2(q)\right)\stackrel{\textrm{Prop. } \ref{step1_lemma2}}{=}$$
$$\mathrm{rank} \left((\mathrm{ad} X)_{\mathrm{mod} D_i}\right) \stackrel{\textrm{Prop. } \ref{step1_lemma1}}{=}\mathrm{rank} \left((\mathrm{ad} X)_{\mathrm{mod} D}|_{D_i(q)}\right)\leq \dim (D_i)^2(q)/D_i(q) $$
Hence, $\dim (D_i)^2(q)/D_i(q)=\dim (D_i)^2(q)/\left(D(q)\cap (D_i)^2(q)\right)$, which implies \eqref{Tanaka_coincidence_1}.

Further, as a direct consequence of Propositions \ref{step1prop} and \ref{D2D3_sec} and the assumption that $D$ is a step $2$ distribution  one gets the following
\begin{Coro}
The Tanaka symbol of $D_i$ at $q$ is isomorphic to $\mathfrak{m}^i(q)$ and the distribution $D_i^2$ is of rank equal to $\dim \mathfrak{m}^i$.
\end{Coro}

To guide a reader, the rest of the proof of Theorem \ref{mainthm}  is organized as follows:  Proposition \ref{step1_lemma1} is proved in subsection  \ref{Prop4.5_sec}, Propositon \ref{Dij^2_involutive} is proved in subsections \ref{3dist_subsec} and \ref{Dij_prop_sec}, Proposition \ref{PropD2D3} is proved in subsection \ref{D2D3_sec}, and Proposition \ref{step1_lemma2} is proved in subsection \ref{Prop4.6_sec}. The final step in the proof of Theorem \ref{mainthm} is done in subsection \ref{rotation_sec}.
\subsection{Proof of Proposition \ref{step1_lemma1}}
\label{Prop4.5_sec}
Let $\mathfrak m^i(q)$, $i=1,2$ be as in \eqref{weakdecompsymb_1}-\eqref{weakdecompsymb_2}. Note that by the paragraph after Proposition \ref{le:constant_alphai}  and the fact that the graded algebras $\mathfrak m^i(q)$ are of step not greater than $2$, $\dim \mathfrak m^i_{j}(q)$ are independent of $q$.  Let 
\begin{equation}
\label{m_id_i_def}
\begin{split}
~&m_i:=\dim \mathfrak m^i_{-1}(q), \quad d_i:=\dim \mathfrak  m^i_{-2}(q);\\
~&n_i:=\sum_{j=1}^{i}m_j,\quad e_i:=\sum_{j=1}^{i}d_j;\\
~&\mathcal{I}_i^1 = \left[\left(n_{i-1}+1\right): n_i\right], \quad \mathcal{I}_i^2 = \left[\left(m+e_{i-1}+1\right):\left(m+e_i\right)\right].
\end{split}
\end{equation}
Note that  $n_0=e_0=0$.

Since $ \mathfrak{m}^i$ is $\mathrm{ad}$-surjective for every $i \in [1:k] $, we can choose a local $g_1$-orthonormal basis $(X_1, \ldots, X_m)$ of $D$ such that the following conditions hold for every $i\in [1:k]$:
\begin{enumerate}
\item $D_i=\mathrm{span} \{X_j\}_{j\in \mathcal{I}_i^1}$;
\item $X_{n_{i-1}+1}(q)$  is an $\mathrm {ad}$-generating  (in a sense of Definition \ref {ad_surj_def}) element of the algebra $\mathfrak m^i(q)$.
\end{enumerate}
Then one can complete $(X_1, \ldots, X_m)$ to the local frame $(X_1, \ldots, X_n)$ of $TM$ by setting  
\begin{equation}
\label{quasi_complete}
       X_{m+e_{i-1}+j}:=[X_{n_{i-1}+1}, X_{n_{i-1}+j+1}], \quad \forall i\in [1:k], j\in[1:d_i]. 
\end{equation}
A local frame $(X_1, \ldots, X_n)$ of $TM$ constructed in this way will be called \emph{quasi-normal frame adapted to the tuple} $\{X_{n_{i-1}+1}\}_{i=1}^k$ of $\mathrm{ad}$-generating elements (one for each $\mathfrak m^i$).

By construction, quasi-normality implies the following conditions for the structure functions of the frame: 
\begin{equation}
\label{quasi_normal_struct}
    c^l_{n_{i-1}+1,j} = \delta_{l,m+j+e_{i-1}-n_{i-1}-1}, \quad \forall i\in[1:k], j\in [n_{i-1}+2: n
    _{i-1}+d_i+1], l\in[1:n],
\end{equation}
where $\delta_{s,t}$ stands for  the Kronecker symbol.

In the sequel, we will work with quasi-normal frames: we start with one quasi-normal frame and, if necessary, perturb it to other quasi-normal frames adapted to the same tuple of $\mathrm{ad}$-generating elements.
The statement of Proposition \ref{step1_lemma1} is true if one shows that $\mathrm{pr}_i$ restricted to $\mathrm{Im}(\mathrm{ad}X)_{\mathrm{mod}D_i}$ is injective while the surjectivity follows automatically from the definition of the projection. 
Without loss of generality, we can assume that $i=1$, as the proof for $i\neq 1$ is completely analogous. The injectivity of   ${\mathrm{pr}_1}|_{\mathrm{Im}(\mathrm{ad}X)_{\mathrm{mod}D_1}}$ is equivalent to 
\begin{equation}
\label{ker0} 
    \ker( {\mathrm{pr}_1}|_{\mathrm{Im}(\mathrm{ad}X)_{\mathrm{mod}D_1}}) =0.
\end{equation}
    Clearly,
\begin{equation}
\label{ker1} 
    \ker( {\mathrm{pr}_1}) = \left((D_1)^2(q) \cap D(q)\right) /D_1(q).
\end{equation}
Set $X=X_1$
Then by \eqref{ker1}
\begin{equation}
    \ker( {\mathrm{pr}_1}|_{\mathrm{Im}(\mathrm{ad}X)_{\mathrm{mod}D_1}}) = \left(\left((D_1)^2(q) \cap D(q)\right) /D_1(q)\right)\cap \bigl(\mathrm{Im}(\mathrm{ad}X)_{\mathrm{mod}D_1}\bigr).
\end{equation}
So, the desired relation \eqref{ker0} is equivalent to  
\begin{equation}
\label{step1_lemma1_struct}
    c^k_{1l} = 0, \text{ for } l \in \mathcal{I}_1^1, k \in \displaystyle{\bigcup_{j=2}^k \mathcal{I}_j^1}.
\end{equation}

In the sequel,  we will use  the following Proposition many times:

 \begin{Prop}[\cite{Z}, Proposition 6]\label{first divi}
    If $g_1$ and $g_2$ are affinely equivalent but not constantly proportional to each other, then the following properties hold:
    \begin{enumerate}
        \item $c^j_{ji}  = 0$, for any $i\in \mathcal{I}^1_s$, $j \in \mathcal{I}^1_v$ with $s\neq v$;
        \item $c^i_{jk} = -c^k_{ji}$, for any $i\neq k \in \mathcal{I}_s^1$, $j \in \mathcal{I}_v^1$  with $s\neq v$;\footnote{In more detail, one of the conclusions of \cite{Z}, Proposition 6], formulated  for projective equivalence, is that $X_{i}\left(\frac{\alpha_j^2}{\alpha_i^2}\right)=2c_{ji}^j \left(1-\frac{\alpha_j^2}{\alpha_i^2}\right)$, but in the case of affine equivalence $\alpha_i^2$ and $\alpha_j^2$ are constant and we obtain that $ c_{ji}^j=0$.}
        \item $(\alpha^2_j - \alpha^2_i)c^k_{ji} + (\alpha^2_j - \alpha^2_k)c^i_{jk} + (\alpha^2_i - \alpha^2_k)c^j_{ik} = 0$ for every pairwise distinct  $i,j,k$ from $[1:m]$. 
    \end{enumerate}
\end{Prop}
Note that item (2) above is the consequence of item (3) applied to the case when one pair in the triple $\{i,j,k\}$ belongs to the same $\mathcal I_s^1$.
In all subsequent lemmas, we will assume that the relations given in items (1)-(3) of Proposition \ref{first divi} hold.

Now let us  give more explicit expressions for the vector $b$ in the fundamental algebraic system \eqref{A.phi.B}, which will be helpful in the sequel:

\begin{Lemma}
\label{b1_lem}
     The entries $b^1_j$ in \eqref{term.of.b.rec.form} with $j\in \mathcal I_i^1$ , are given by 
    \begin{equation}\label{b in first group}
        b^1_j = (\alpha_{n_{i-1}+1})^2 \sum_{l\in\mathcal I_i^2} q_{jl}u_l + \sum _{s\neq i}\left( (\alpha_{n_{i-1}+1})^2 -(\alpha_{n_{s-1}+1})^2\right) \sum_{l\in\mathcal I_s^1} q_{jl}u_l,
    \end{equation}
    where $q_{jk}$ are defined by \eqref{qjk}.
\end{Lemma}
\begin{proof}
Using \eqref{eq:vech}, \eqref{poison1}, and \eqref{struct_function}, we get  
    \begin{equation}
    \label{huj}
        \Vec{h}_1(u_j) = \sum_{i=1}^{m} u_i \vec u_i(u_j)= 
        \sum_{k=1}^n \sum_{i=1}^m u_i u_k c^k_{ij} = \sum_{k=1}^n q_{jk}u_k.
    \end{equation}

    
    From \eqref{term.of.b.rec.form}  and \eqref{RJ}, using \eqref{qjk} and the fact that  by our constructions $(\alpha_l)^2=(\alpha_{n_{i-1}+1})^2$ for every $l\in \mathcal I_i^1$, we have
    \begin{equation}
    \label{b in first group_1}
        \begin{split}
        b^1_j &= (\alpha_{n_{i-1}+1})^2 \Vec{h}_1(u_j) - \sum_{1\leq i,l\leq m} (\alpha_l)^2 u_i u_l c^l_{ij} \\
        &= (\alpha_{n_{i-1}+1})^2 \sum_{k=1}^n q_{jk}u_k  
        - \sum_{i=1}^k (\alpha_{n_{i-1}+1})^2\sum_{l\in \mathcal I_i^1} q_{jl} u_l.
        \end{split}
    \end{equation}
    Note  that \eqref{Tanaka_decomposition} we have \eqref{qjk} $q_{jk} = 0$ for $k\in \displaystyle{\bigcup_{s\neq i} \mathcal{I}_s^2} $, so 
\begin{equation}
\label{b in first group_2}
\sum_{k=1}^n q_{jk}u_k= \sum _{s=1}^k\sum_{l=\mathcal I_s^1}^m q_{jl}u_l+ \sum_{l\in\mathcal I_i^2} q_{jl}u_l.
\end{equation}
Substituting \eqref{b in first group_2} into \eqref{b in first group_1}, we get \eqref{b in first group}.
\end{proof}



Lemma \ref{b1_lem} implies  that to analyze  maximal minors of the augmented matrix $[A|b]$ it is convenient to perform the
following column operation by setting 
\begin{equation}
\label{column_op}
    \Tilde{b} = b - \sum_{i=1}^k(\alpha_{n_{i-1}+1})^2 \sum_{t=e_{i-1}}^{e_i} (A)_t u_{m+t}, 
\end{equation}
where $(A)_j$ represents the $j$th column of $A$ and $e_i$ are defined in \eqref{m_id_i_def}, so the corresponding maximal minors of $[A|b]$ and $[A|\Tilde b]$ coincide. 
\begin{rk}
\label{b>m_rem}
\emph{Using the first line of \eqref{elem.A}, one gets that the first term in \eqref{b in first group} 
is canceled by the column operations, namely
\begin{equation}
\label{b^{1,2}_i}
  \Tilde b_j^1 = \sum _{v\neq i}\left( (\alpha_{n_{i-1}+1})^2 -(\alpha_{n_{v-1}+1})^2\right) \sum_{l\in\mathcal I_v^1} q_{jl}u_l, \quad j\in \mathcal I_i^1,
     \end{equation}
where $\Tilde b^1_j$ is the $j$th component (from the top) of the column vector $\Tilde b^1$.   
Hence, $\Tilde b^1$ has no term of $u_i$'s, with $i>m$.}
\end{rk}

\begin{Lemma}
\label {b^1_the same}
Let $i, v\in [1:k]$, $i\neq v$, $j\in \mathcal I_i^1$ and $r, l\in \mathcal I_v^1$. Then $\Tilde b_j$ does not contain a monomial $u_r u_l$ and in particular it does not contain squares $u_r^2$.
\end{Lemma}
\begin{proof}
Indeed , by \eqref{b^{1,2}_i}, using \eqref{qjk}, the coefficient of the monomial  $u_r u_l$ in is equal to 
\begin{equation}
\Bigl(\bigl(\alpha_{n_{i-1}+1})^2 -(\alpha_{n_{v-1}+1}\bigr)^2\Bigr)\Bigl(c_{rj}^l+c_{lj}^r\Bigr),
\end{equation}
which is equal to zero by items (1) and  (2) of Proposition \ref{first divi}. 
\end{proof}

Now we will make a long analysis of coefficients of specific monomials in specific $(n-m+1)\times (n-m+1)$ minor of the augmented matrix $[A|b]$ (equivalently, $[A|\Tilde b]$).

First,  to achieve \eqref{step1_lemma1_struct}, given $i_0\in [1:k]$ consider submatrices $M_{i_0}$ of $[A|\Tilde b]$  consisting of rows with indices from the set 
\begin{equation}
\label{rows_M1}
S_{i_0}:=[n_{i_0-1}+1: n_{i_0-1}+d_{i_0}+1] \cup \bigcup_{i\in [1:k]\backslash \{i_0\}}
[n_{i-1}+1: n_{i-1}+d_i].   
\end{equation}
From \eqref{Tanaka_decomposition} it follows that  $M_{i_0}$ is a block-diagonal matrix,
\begin{equation}
\label{M_i}
    M_{i_0} =\begin{pmatrix}
            {M_{i_0,1}} & 0 & 0 \\
            0 & \ddots & 0 \\
            0 & 0 & M_{i_0, k}
        \end{pmatrix},
\end{equation}
where $M_{i_0,i_0}$ is of size $(d_{1_0}+1)\times d_{i_0}$ and $M_{i_0,i}$, $i\neq i_0$ is of size $d_{i_0}\times d_{i_0}$. 
\begin{rk}
\label{M_indep_rem}
\emph{Note that given $j\in [1:k]$ the blocks $M_{i,j}$ are the same for all $i\neq j$. } 
\end{rk}
Let $b^{1}|_{S_{i_0}}$ (resp. $\Tilde b^{1}|_{S_{i_0}}$) be the subcolumn  of the column $b^1$ (reps. $\Tilde b^1$) consisting of the same rows as in $M_{i_0}$, i.e., the rows of $b$ (resp. $\Tilde b$)  from the set $S_{i_0}$ as in  \eqref{rows_M1}. Since the fundamental algebraic system is an overdetermined linear system admitting a solution, the determinant $\det ([M_{i_0}|b^1|_{S_{i_0}}])$ must vanish, as a polynomial with respect to $u_i$'s. It implies that the coefficients of each monomial w.r.t $u_i$'s in $\det (M_1|b^{1,1})$ must vanish as well. 
We have the following 
\begin{Lemma}
\label{lemma step 1}
 Given $i_0 \in [1:k]$, if the coefficients of all monomials of the form 
\begin{equation}\label{monomial 1}
   u_l u_s \left(\prod_
   {i\neq i_0}u_{n_{i-1}+d_i+1}\right) (u_{n_{i_0-1}+1})^{d_{i_0}}\left(\prod_
    {i\neq i_0}
   (u_{n_{i-1}+1})^{d_i-1}\right),
 \quad l\in \mathcal{I}_{i_0}^1, 
 s\in [1:m] \backslash\mathcal I_{i_0}^1
\end{equation}
in $\det ([M_{i_0}|b^{1}|_{S_{i_0}}])$  vanish,
then 
\begin{equation} \label{c^k}
    c^s_{n_{i_0-1}+1\,l} = 0, \quad l\in \mathcal{I}_{i_0}^1, s\in [1:m]\backslash \mathcal I_{i_0}^1.
\end{equation}
\end{Lemma}


\begin{proof}

In the sequel, we will refer to the classical formula for determinants in terms of permutations of the matrix elements as the \emph{Leibniz formula for determinants}. 
Without loss of generality, we can assume that $i_0=1$ and $s\in \mathcal I_2^1$.
Using  \eqref{qjk}, the first line of \eqref{elem.A}, and \eqref{quasi_normal_struct} it is easy to conclude that  
the variable $u_{n_{i-1}+1}$ appears in the following  columns of the augmented matrix $[M_1|b^{1}|_{S_{1}}]$ only:
\begin{enumerate}
\item[(A$1_i$)] The columns containing all columns of the matrix $M_{1,i}$ if $i=1$  or all columns of $M_{1,i}$ except the last one, if $i\in [2:k]$ . Moreover,  in each of these columns, $u_{n_{i-1}+1}$ appears exactly in the entry $(M_{1, i})_{j+1 j}$, i.e., in the entry situated right below the diagonal of the block $M_{1, i}$. Besides, the coefficient of $u_{n_{i-1}+1}$ in this entry is equal to $1$;
\item[(A$2_i$)] the last column of the matrix  $[M_1|b^{1}|_{S_{1}}]$.
\end{enumerate}

Applying the   normalization conditions \eqref{quasi_normal_struct} to   \eqref{b^{1,2}_i} one gets
\begin{enumerate}
\item [(A$3_i$)] The components  of the column vector ${b}^1|_{S_{1}}$  in the rows corresponding to the rows  of the block $M_{1, i}$ of $M_1$ do not contain $u_{n_{i-1}+1}$.
\end{enumerate}

For our purpose, it is enough to set variables $u_i$ not appearing in \eqref{monomial 1} to be equal to $0$. 
From item (A$1_i$)  above,
the fact that $u_{n_{i-1}+1}$ appears in the power  not less than  $d_i-1$ in the monomial \eqref{monomial 1}, and that by Lemma \ref{b^1_the same} and (A$3_i$) the participating entries do not contain $(u_{n_{i-1}+1})^2$, it follows that in the Leibniz formula for the determinant of the matrix $(M_1|\Tilde b^{1}|_{S_{1}})$ the contribution to the monomial 
\eqref{monomial 1} from the block $M_{1,i}$ comes only from the following terms:
\begin{enumerate}
\item[(B$1_i$)] \emph{Terms containing  all factors of the form $ (M_{1,i})_{j+1 \,j}$.} Moreover, from the facts that  all rows of $M_{1,i}$ except the first one are used, that  $M_1$ has block diagonal structure, and that 
\begin{enumerate}
\item for  $i=1$ all columns appearing in $M_{1, i}$ are used,  it follows that in this case the terms giving the desired contribution must contain the factor $b_1^1$ as the only possible factor from the first row of the augmented matrix  $(M_1|b^{1}|_{S_{1}})$.
\item for $i\in [2:k]$  all columns of $M_{1,i}$ except the last one
are used, it follows that in this case the terms giving the desired contribution must contain the entry $(M_{1, i})_{1\, d_2}$. Note that if we take into account only those $u_i$'s which  appear in \eqref{monomial 1}  (or, equivalently, set all other $u_i$'s equal to zero) and that, by \eqref{quasi_normal_struct}, $c^{m+e_1}_{n_1+1, s}=\delta_ {n_1+d_2+1, s}$, we have
\begin{equation}
\label{M+1,2_top_right}
(M_{1,i})_{1\, d_i}=
%
- u_{n_{i-1}+d_i+1}
.
\end{equation}
\end{enumerate} 
\item [(B$2_i$)] (possible only if either $i=1$ or  $i=2$ and $s\neq n_1+1$ \footnote{Otherwise the desired power of $u_{n_{i-1}+1}$ in \eqref{monomial 1} cannot be achieved.}) \emph{Terms containing all factors of the form  
$(M_{1,i})_{j+1 \,j}$, $j\in[1: d_2-1]$ except one.} Then these terms also contain a factor from the column $b^{1}|_{S_{1}}$ depending on the variable $u_{n_{i-1}+1}$.
\end{enumerate}

Now consider four possible cases separately:

{\bf (C1) Assume that  (B$2_1$) and (B$2_2$) occur simultaneously.} 
Since for every $i\in [3:k]$ item $(B1_i)$ holds,   the participating factor $b^1_j$ must satisfy  $j\in [1:d_1+1] \cup [n_1+1:n_1+d_2]$. On the other hand, it must contain  the monomial $u_1u_{n_1+1}$, which by $(A3_1)$ and $(A3_2)$ implies that $j\notin  [1:d_1+1] \cup [n_1+1:n_1+d_2]$, so we got the contradiction. 
So, the considered term does not contribute to the monomial \eqref{monomial 1}.

{\bf  (C2) Assume that  (B$1_i$) for all $i\in[2:k]$ and (B$2_1$) occur simultaneously.} 
Then the participating factor $b^1_j$ from $b^{1}|_{S_{1}}$ must, on one hand, satisfy $j\in [1:d_1]$  and, on the other hand, must contain $u_1$, which contradicts  $(A3_1)$.
So, the considered term does not contribute to the monomial \eqref{monomial 1}.

{\bf  (C3) Assume that  (B$1_i$) for $i\neq 2$ and (B$2_2$) occur simultaneously.}
In this case the contribution to the monomial
\eqref{monomial 1} is from the coefficient of the monomial $u_l u_{n_1+1}$ in the factor $b_1^1$, which is equal to $((\alpha_{n_1+1})^2 - \alpha_1^2)c_{1l}^{n_1+1}$.

 {\bf (C4). Assume that  (B$1_i$) occur simultaneously for every $i\in[1:k]$.} In this case, the  coefficient of  the monomial
 \eqref{monomial 1}  is equal, up to a sign,  to the coefficient of the monomial $u_l u_s u_{n_1+d_2+1}$ in the polynomial   $b_1^1 (M_{1,2})_{1, d_2}$, because the coefficients of the relevant monomials in all other factors in the corresponding term of Leibniz formula are equal to 1. From \eqref{M+1,2_top_right} it follows that the coefficient of the monomial $u_l u_s u_{n_1+d_2+1}$ in the polynomial   $b_1^1 (M_{1,2})_{1, d_2}$ is equal to 
 the coefficient of $u_l u_s$ in $b_1^1$,
 which by \eqref {b^{1,2}_i}  is equal to 
\begin{equation}
\label{C4_total}
( (\alpha_{n_1+1})^2-\alpha_1^2)
c_{1l}^{s}.
\end{equation}

Now we are ready to complete the proof of the lemma.  First, assume that $s=n_1+1$. Then the case ($B2_2$) and therefore (C3) is impossible, so (C4) holds and  the  coefficient of  the monomial
 \eqref{monomial 1}  is equal, up to a sign,  to the expression in \eqref{C4_total} 
 with $s=n_1+1$, so vanishing this coefficient implies  
\eqref{c^k} for $s=n_1+1$. 

Further, the last paragraph implies the case (C3) does not contribute to the coefficients of the monomials \eqref{monomial 1}, so the only contribution is from the case (C4), which is equal to the expression in \eqref{C4_total}.
Vanishing of the latter implies \eqref{c^k}, which completes the proof of Lemma \ref{lemma step 1}.
\end{proof}


Now given $i\in [1:k]$, assume that $X_{n_{i-1}+1}$ is a local section of $D_i$  such that $X_{n_{i-1}+1}(q)$ is an  $\mathrm{ad}$-generating element of $\mathfrak m^i(q)$ in the sense of Definition \eqref{ad_surj_def} for every $q$. Let 
\begin{align}
~& K_{
X_{n_{i-1}+1}}
:=\mathrm {ker} \left((\mathrm{ad}\, X_{n_{i-1}+1})_{\mathrm{mod} D_i}\right) \label{K_def}\\
~& H_{X_{n_{i-1}+1}}:=(\mathrm{span} \{X_{n_{i-1}+1}\})^\perp \cap D_i, \label{H_def} 
\end{align}
where $~^\perp$ stands for the $g_1$-orthogonal complement. 
As a direct consequence of Lemma \ref{lemma step 1}, we get the following
\begin{Coro}
\label{indep_cor}
For any $d_i$-dimensional subspace $F_i$ of $H_{X_{n_{i-1}+1}}$ with
\begin{equation}
\label{FK}
F_i\cap K_{X_{n_{i-1}+1}}=0    
\end{equation}
the image of the restriction of the map $(\mathrm{ad}\, X_{n_{i-1}+1})_{\mathrm{mod} D_1}$ to the subspace $F_i$ coincides with the entire image of the map $(\mathrm{ad}\, X_{n_{i-1}+1})_{\mathrm{mod} D_i}$,
\begin{equation}
\label{Im_indep_F}
\mathrm{Im}  \Bigl(\bigl(\mathrm{ad}\, X_{n_{i-1}+1})_{\mathrm{mod} D_i}\bigr)|_{_{F_i}}\Bigr)=\mathrm{Im}  \bigl((\mathrm{ad}\, X_{n_{i-1}+1})_{\mathrm{mod} D_i}\bigr).
\end{equation}
In particular, this image of the restriction is independent of $F_i$. 
\end{Coro}
\begin{proof}
Indeed, previously we used $\mathrm{span}\{X_{n_{i-1}+1}, \ldots, X_{n_i}\}$ as a subspace $F_i$ but relation \eqref{FK} was the only property we actually used to get the conclusion of Lemma \ref{lemma step 1}.
\end{proof} 

We will denote the space in \eqref{Im_indep_F} by $\mathcal L_{X_{n_{i-1}+1}}$,
\begin{equation}
 \label{Li}
 \mathcal L_{X_{n_{i-1}+1}}:=
\mathrm{Im}  \bigl((\mathrm{ad}\, X_{n_{i-1}+1})_{\mathrm{mod} D_i}\bigr).
\end{equation}

\begin{rk}
\label{d=0_rem}
\emph{Note that if $d_i=0$, then  any element of $D_i$ is trivially $\mathrm{ad}$-generating element of $\mathfrak m_i$.}
\end{rk}
\begin{rk}
\label{d=1_rem}
\emph{If $d_i=1$, then $\dim K_{X_{n_{i-1}+1}}=\dim H_{X_{{n_{i-1}+1}}}=m_i-1$ but  $K_{X_{n_{i-1}+1}}\neq H_{X_{{n_{i-1}+1}}}$, because $X_{{n_{i-1}+1}}$ belongs to $K_{X_{n_{i-1}+1}}$ but does not belong to  $H_{X_{n_{i-1}+1}}$, so  $K_{X_{n_{i-1}+1}}\cap H_{X_{n_{i-1}+1}}$ is a codimension $1$ subspace $H_{X_{n_{i-1}+1}}$. Therefore, we can perturb the original quasi-normal frame $(X_1, \ldots, X_n)$ to a quasi-normal frame  $(\Tilde X_1, \ldots, \Tilde X_n)$  adapted to the same tuple of $\mathrm{ad}$-generating elements as the original frame   such that for every $i\in [1:k]$ and $j\in \mathcal I_i^1$ the line generated by $\Tilde X_j$ is transversal to $K_{X_{n_{i-1}+1}}$. Since $d_i=1$ and by Corollary \ref{indep_cor} $[X_{n_{i-1}+1}, \tilde X_j]\neq 0 \,\mathrm{mod} D$ (otherwise $\Tilde X_j\in K_{X_{n_{i-1}+1}}$), we get that $\tilde X_j$ is also $\mathrm{ad}$-generating element of $\mathfrak{m}_i$.}
\end{rk}
\subsection{Proof that $[D_r, D_t]\in \mathcal L_{X_{n_{r-1}+1}}+\mathcal L_{X_{n_{t-1}+1}} \quad \mathrm{mod}\,\, D_r+D_t\,\, r\neq t$ (toward the proof of Proposition \ref{Dij^2_involutive})}
\label{3dist_subsec} The claim in the title of the subsection is equivalent to 
\begin{equation}
\label{3dist_0} 
\begin{split}
~&c_{jl}^s=0
\quad \forall j\in \mathcal I_r^1, s\in \mathcal I_t^1,  l\in\mathcal I_i^1,\\
~& \text {where  } \{i, r, t\}\in [1:k] \text { are pairwise distinct.}
\end{split}
\end{equation}
In particular, for $k\leq 2$ equation \eqref{3dist_0} is void, so it is relevant to assume that $k\geq 3$.

 Given $i_0\in [1:k]$ consider submatrices $P_{i_0}$ of $[A|\Tilde b]$  consisting of rows with indices from the set 
\begin{equation}\label{rows of N_1}
  R_{i_0}:= \left( \bigcup_{i\in [1:k]\backslash \{i_0\}} [n_{i-1}+2: n_{i-1}+d_i+1] \right) \cup [n_{i_0-1}+1: n_{i_0-1}+d_{i_0}+1].
\end{equation}
From \eqref{Tanaka_decomposition} it follows that  $P_{i_0}$ is a block-diagonal matrix,
\begin{equation}
\label{P_i}
    P_{i_0} =\begin{pmatrix}
            {P_{i_0,1}} & 0 & 0 \\
            0 & \ddots & 0 \\
            0 & 0 & P_{i_0, k}
        \end{pmatrix},
\end{equation}
where $P_{i_0,i_0}$ is of size $(d_{1_0}+1)\times d_{i_0}$ and $P_{i_0,i}$, $i\neq i_0$ is of size $d_{i_0}\times d_{i_0}$. Note that by construction
\begin{equation}
 \label{PM}
 P_{i_0, i_0}=M_{i_0, i_0},
\end{equation}
where $M_{i_0,i_0}$ is defined in \eqref{M_i}.

Similar to the previous subsection, let $\Tilde b^{1}|_{P_{i_0}}$ be the subcolumn  of the column $\Tilde b^1$ consisting of the same rows as in $P_{i_0}$, 
\begin{Lemma}\label{lemma step 1 k com}
If for every triple  of pairwise distinct integers $\{i_0, r,t\} \in [1:k]$ the coefficient of all monomials of the form
\begin{equation}\label{monomial 4 k com}
\begin{split}
~&   u_j u_l u_s u_{n_{i_0-1}+1}^{d_{i_0}-1} \prod_{i\in [1:k]\backslash \{i_0\}}u_{n_{i-1}+1}^{d_i}, \quad \text{where  } j\in  [n_{r-1}+1:n_{r-1}+d_r+1], \\~&l\in [n_{t-1}+1:n_{t-1}+d_t+1], s\in [n_{{i_0}-1}+1:n_{i_0-1}+d_{i_0}+1]
\end{split}
\end{equation}
    in the determinant $\det [P_{i_0}| \Tilde b|_{R_{i_0}}]$ vanish, then 

  \begin{equation}
   \label{struct_3_distinct}
  \begin{split}
  ~& c_{jl}^s=0, \text{where  } j\in  [n_{r-1}+1:n_{r-1}+d_r+1], \\~&l\in [n_{t-1}+1:n_{t-1}+d_t+1], s\in [n_{{i_0}-1}+1:n_{i_0-1}+d_{i_0}+1]
  \end{split}
  \end{equation}  
 or, equivalently, 
 \begin{equation}
\label{struct_3_distinct_reform}
\begin{split}
~&[X_j, X_l]\in \mathcal L_r+\mathcal L_t\quad \mathrm{mod} (D_r+D_t), \\
~&j\in  [n_{r-1}+1:n_{r-1}+d_r+1], l\in [n_{t-1}+1:n_{t-1}+d_t+1].
\end{split}
\end{equation} 
\end{Lemma}

\begin{proof}
Without loss of generality, we can assume that $r=1$, $t=2$, and $i_0=k$. Further,  using \eqref{PM},
similar to  the proof of Lemma \ref{lemma step 1} (statements of  items ($A1_i$) and $(A2_i)$ there), one can conclude that  
the variable $u_{n_{i-1}+1}$, $i\in [1:k]$ appears in the following  columns of the augmented matrix $[N_1|\Tilde b|_{T_{1}}]$ only:
\begin{enumerate}
\item[(D$1_i$)] The columns containing all columns of $P_{k,i}$ if $i\in [1:k-1]$, and all columns of $P_{k,i}$ except the last one, if $i=k$. Moreover,  in each of these columns $u_{n_{i-1}+1}$ appears exactly in the diagonal entry $(P_{k, i})_{j j}$  if $i\in [1:k-1]$, and in the entry $(P_{k, i})_{j+1 j}$,  i.e., in the entry situated right below the diagonal of the block $P_{k, i}$, if $i=k$. Besides, the coefficient of $u_{n_{i-1}+1}$ in this entry is equal to $1$;
\item[(D$2_i$)] the last column of the matrix $[ P_k|\Tilde b|_{R_{k}}]$;
\end{enumerate}
 As before, we can set all irrelevant variables $u$'s (which are not in \eqref{monomial 4 k com}) to zero. From this and the normalization condition \eqref{quasi_normal_struct} it follows that
 \begin{enumerate}
 \item [(D3)]   The only nonzero entry of the first row of $P_{k,k}$ is the entry in the $s-n_{k-1}-1$ column, i.e. $P_{k,k}|_{1, s-n_{k-1}1}$ and it is equal to $-u_s$.   
 \end{enumerate}
 Finally, the statement of $(A_3)_i$ with $M_1$ replaced by $P_k$, $M_{1,i}$ replaced by $P_{1,i}$, and $S_1$ replaced by $R_k$ holds true, i.e.,
\begin{enumerate}
 \item [(D$4_i$)] The components  of the column vector ${b}^1|_{R_{1}}$  in the rows corresponding to the rows  of the block $P_{k, i}$ of $P_k$ do not contain $u_{n_{i-1}+1}$.
 \end{enumerate}
 
 From item ($\mathrm{D}1_i$)  above, 
the fact that $u_{n_{i-1}+1}$ appears in the monomial \eqref{monomial 4 k com}  in the power $d_i$ for $i\in [1:k-1]$ and in the power not less than  $d_i-1$ if $i=k$, and that by Lemma \ref{b^1_the same} and ($D4_i$) the participating entries do not contain $(u_{n_{i-1}+1})^2$ it follows that in the Leibniz formula for the determinant of the matrix $(P_k|\Tilde b^{1}|_{R_{k}})$ the contribution to the monomial 
\eqref{monomial 4 k com}  from the block $P_{k,i}$ comes only from the following terms:
\begin{enumerate}
\item[(E1)] \emph{Terms containing  all factors of the form $ (P_{k,i})_{j\,j}$, $j\in [1:d_i]$, if $i\in[1:k-1]$.} Otherwise, if one of them is omitted, then a non-diagonal entry of $P_{k, i}$ has to be used as well, therefore another diagonal term of  $P_{k, i}$ has to be omitted, however in this way there is no way to reach  $u_{n_{i-1}}$ in the power $d_i$ in the resulting monomial. 
\item[(E2)] \emph{Terms containing  all factors of the form $ (P_{k,i})_{j\,j+1}$, $j\in [1:d_k]$,  if $i=k$ and $s=n_{k-1}+1$.} 
In this case, the desired contribution must contain the component of $\Tilde b^1$ which belongs to the row, corresponding to the first row of $P_{k,k}$, i.e., $\Tilde b_{n_{k-1}+1}$. 
 
\item [(E3)] 
\emph{Terms containing all factors of the form  
$(P_{k,k+1})_{j+1 \,j}$, $j\in[1: d_k$ except one, if $i=k$ and $s\neq n_{k-1}+1$ .}Then by (D3) and the fact that $P_k$ has the block-diagonal form as in \eqref{P_i} it follows that the omitted factor is $(P{k,k}_{s-n_{k-1},s-n_{k-1}-1}$. Hence,  the desired contribution must contain the component of $\Tilde b^1$ which belongs to the row, corresponding to the $s-n_{k-1}$th row of $P_{k,k}$, i.e., the component $\Tilde b_s$. \end{enumerate}
So, we have two cases
\begin{enumerate}
\item [(F1)] $s=n_{k-1}+1$. In this case, by combining (E1) and (E2),  we conclude that 
the coefficient of the monomial
\eqref{monomial 4 k com} in  $\det [P_{i_0}| \Tilde b|_{R_{i_0}}]$ is, up to a sign, equal to the coefficient of the monomial $u_ju_l$ in $\Tilde b_{n_{k-1}+1}$, which by \eqref{b^{1,2}_i} is equal to  
 \begin{equation}
 \label{3dist_1}
        \left( \aa_k^2 - \aa_1^2 \right) c^j_{l,n_{k-1}+1} + \left( \aa_k^2 - \aa_2^2 \right)c^l_{j,n_{k-1}+1}.
    \end{equation}

\item[(F2)] $s\neq n_{k-1}+1$.
In this case, by combining (E1) and (E3) and using \eqref{b^{1,2}_i}, we conclude   that 
the coefficient of the monomial
\eqref{monomial 4 k com} in  $\det [P_{i_0}| \Tilde b|_{R_{i_0}}]$ is, up to a sign, equal to the coefficient of the monomial $u_ju_l$ in $\Tilde b_{s}$, which by \eqref{b^{1,2}_i} is equal to   
\eqref{monomial 4 k com}
\begin{equation}
\label{3dist_2}
\left( \aa_k^2 - \aa_1^2 \right) c^j_{l,s} + \left( \aa_k^2 - \aa_2^2 \right)c^l_{j,s}=0.
\end{equation}
\end{enumerate}
So, by the assumptions of the lemma expressions in \eqref{3dist_1} and \eqref{3dist_2} vanish.

Repeating the same arguments for arbitrary pairwise distinct triple $\{r, s, i_0\}\in [1:k]$ instead of $\{1,2,k\}$ we will get that
\begin{equation}
\label{3dist_3}
\begin{split}
        ~&\left( \aa_{i_0}^2 - \aa_r^2 \right) c^j_{l,s} + \left( \aa_{i_0}^2 - \aa_t^2 \right)c^l_{j,s}=0, \quad  \forall j\in  [n_{r-1}+1:n_{r-1}+d_r+1], \\~&l\in [n_{t-1}+1:n_{t-1}+d_t+1], s\in [n_{{i_0}-1}+1:n_{i_0-1}+d_{i_0}+1].
\end{split}
\end{equation}
Permuting the indices  in \eqref{3dist_3} we can get
\begin{equation}
\label{3dist_4}
\left( \aa_t^2 - \aa_r^2 \right) c^j_{s,l} + \left( \aa_t^2 - \aa_{i_0}^2 \right)c^s_{j,l}=
\left( \aa_r^2 - \aa_t^2 \right) c^j_{l,s} + \left( \aa_t^2 - \aa_{i_0}^2 \right)c^s_{j,l}=
0.
\end{equation}

Finally  item (3) of Proposition \ref{first divi} implies
\begin{equation}
\label{item3_first_div}
\left( \aa_{i_0}^2 - \aa_t^2 \right) c^j_{s,l} + \left( \aa_r^2 - \aa_{i_0}^2 \right)c^l_{j,s}+\left( \aa_r^2 - \aa_{t}^2 \right)c^s_{j,l}=0.
\end{equation}

 The linear homogeneous system with respect to $c^{j}_{l,s}$, $c^{l}_{j, s}$, and $c^{s}_{j, s}$
    consisting of equations  \eqref{3dist_3}, \eqref{3dist_4}, and \eqref{item3_first_div} has the matrix
    \begin{equation}
    \begin{pmatrix}
    \aa_{i_0}^2-\aa_r^2&\aa_{i_0}^2-\aa_t^2&0\\\aa_r^2-\aa_t^2&0
    &\aa_t^2-\aa_{i_0}^2\\ \aa_{i_0}^2-\aa_t^2&
    \aa_r^2-\aa_{i_0}^2&\aa_r^2-\aa_t^2
    \end{pmatrix},    
    \end{equation}
    whose determinant is equal to
    $$(\aa_t^2-\aa_{i_0}^2)\Bigl(
    (\aa_{i_0}^2-\aa_j^2)^2+(\aa_r^2-\aa_t^2)^2+(\aa_{i_0}^2-\aa_t^2)^2\Bigr)$$
  and is not zero as by our assumption that  $\aa_{i_0}^2$,  $\aa_r^2$, and  $\aa_t^2$ are pairwise distinct.
  This implies \eqref{struct_3_distinct}
  \end{proof}

Lemma \ref{lemma step 1 k com} 
proves only a subset of relations from \eqref{3dist_0}.   Using the flexibility given by Corollary \ref{indep_cor} one can show that \eqref{3dist_0} holds not for the original quasi-normal frame but for its perturbation adapted to the same tuple of $\mathrm{ad}$-generating vector fields $\{X_{n_{i-1}+1}\}_{i=1}^k$.

\begin{Coro}
\label{ils_cor}
Let $K_{X_{n_{i-1}+1}}$ and $H_{X_{n_{i-1}+1}}$ be as in \eqref{K_def} and \eqref{H_def}, respectively.  For every $i\in [1:k]$ let $Y_i$  be either equal to $X_{n_{i-1}+1}$ or  a normalized (i.e., $g(Y_i, Y_i)=1)$ local section of $H_{X_{n_{i-1}+1}}$ so that
\begin{equation}
\label{Y_trans_K} 
Y_i\notin K_{X_{n_{i-1}+1}}.
\end{equation}
Then for every $r\neq t$ from $[1:k]$
\begin{equation}
\label{struct_3_distinct_reform_1}
[Y_r, Y_t]\in \mathcal L_r+\mathcal L_t\quad \mathrm{mod} (D_r+D_t).
\end{equation}
\end{Coro}
\begin{proof}
Indeed, \eqref{Y_trans_K} implies that we can find a quasi-normal frame $(X_1, \ldots, X_n)$ such that either $X_{n_{i-1}+1}=Y_i$ or  $X_{n_{i-1}+2}=Y_i$. Then \eqref{struct_3_distinct_reform_1} follows from \eqref{struct_3_distinct_reform} of lemma \ref{lemma step 1 k com}.
\end{proof}

If $d_i\geq 1$, then $\dim K_{X_{n_{i-1}+1}}\geq \dim H_{X_{{n_{i-1}+1}}}$ and $K_{X_{n_{i-1}+1}}\neq H_{X_{{n_{i-1}+1}}}$. The latter holds because $X_{{n_{i-1}+1}}$ is in  $K_{X_{n_{i-1}+1}}$ but not in $H_{X_{n_{i-1}+1}}$. So, in this case the complement of  $K_{X_{n_{i-1}+1}}\cap H_{X_{{n_{i-1}+1}}}$ to $H_{X_{{n_{i-1}+1}}}$  is open and dense in $H_{X_{{n_{i-1}+1}}}$. Therefore, by a finite number  of consecutive small perturbations of the original quasi-normal frame $(X_1, \ldots, X_n)$,  one can build a quasi-normal frame  $(\Tilde X_1, \ldots, \Tilde X_n)$  adapted to the same tuple of $\mathrm{ad}$-generating elements as the original frame   such that for every $r\in [1:k]$ 
\begin{itemize}
\item if $d_r\geq 1$, then every $j\in \mathcal I_r^1\backslash \{n_{r-1}+1\}$, the vector field  $\tilde X_j$  is in the complement of  $K_{X_{n_{i-1}+1}}\cap H_{X_{{n_{i-1}+1}}}$ to $H_{X_{{n_{i-1}+1}}}$;
\item if $d_r=0$, then no additional conditions on $\tilde X_j$ with $j\in \mathcal I_r^1$ are impose and by permutation any $\tilde X_j$ can be seen as $X_{n_{r-1}+1}$ (note that in this case any element of $D_r$ is trivially $\mathrm{ad}$-generating element of $\mathfrak m_r$). 
\end{itemize} 
So, by \eqref{struct_3_distinct_reform_1} this frame satisfies
\begin{equation}
\label{struct_3_distinct_reform_2}
[X_j, X_l]\in \mathcal L_r+\mathcal L_t\quad \mathrm{mod} (D_r+D_t), 
j\in  \mathcal I_r^1, l\in \mathcal I_t^1, \quad r\neq t\in [1:k],
\end{equation}
which is equivalent to
\eqref{3dist_0}. 

\begin{rk}
\label{pert_indep_remark} 
\emph{Note that if \eqref{struct_3_distinct_reform_2} holds for some quasi-normal frame, then it holds for any quasi-normal frame adapted to the same tuple of $\mathrm{ad}$-generating elements as the original one.}
\end{rk}

As the direct consequence of \eqref{b^{1,2}_i} and \eqref{3dist_0}, we get that  for any triple $\{i, r, t\}$ of pairwise distinct integers from $[1:k]$  and any  $j\in \mathcal I_i^1$, $j\in \mathcal I_r^1$, and $l\in\mathcal I_t^1$, the polynomial $\Tilde b_j$ does not contain a monomial $u_r u_t$. Combining this with Lemma \ref{b^1_the same}, we get 

\begin{Lemma}
\label {b^1_the same_1}
Given $j\in \mathcal 
I_i^1$, 
$\Tilde b_j$ does not contain monomials $u_r u_l$ with $r, l\in [1:m]\backslash \mathcal I_i^1$, or,equivalently, every monomial of $\Tilde b_j$ must contain a variable $u_r$ with $r\in \mathcal 
I_i^1$.
\end{Lemma}
Moreover, from \eqref{b^{1,2}_i} and Lemma \ref{lemma step 1} it follows that 
\begin{Coro}
\label{first_bi_coro}
  For every $i\in [1:k]$
  \begin{equation}
   \label{first_b_i}
   \Tilde b_{n_{i-1}+1}=0.
  \end{equation}
\end{Coro}
\subsection{Proof of Proposition \ref{step1_lemma2}}
\label{Prop4.6_sec}
The statement of Proposition \ref{step1_lemma2} is equivalent to showing that
\begin{equation}
        D_i^2(q) \cap \bigoplus_{j\in[1:k]\backslash\{i\}}D_j(q) = \{0\}. 
\end{equation} 
In terms of structure functions of a quasi-normal frame, it is equivalent to showing that
\begin{equation}\label{target 2}
    c^s_{lr} = 0, \text{ for } l, r \in \mathcal{I}_i^1, \, s \in [1:m]\backslash \mathcal{I}_i^1 .
\end{equation}

\begin{rk}
\label{d11_rem}
\emph{Note that by Remarks \ref{d=0_rem} and \ref{d=1_rem}, if $d_i=0$ or $d_i=1$ for some $i\in [1:k] $
one can perturb the original quasi-normal frame to a quasi-normal frame $(\Tilde X_1, \ldots, \Tilde X_n)$ for which all $\Tilde X_j$ are $\mathrm{ad}$-generating elements of $\mathfrak m_i$ (for $d_i=0$ any quasi-normal frame satisfy this property). This implies that for this (maybe perturbed) frame  relation \eqref{target 2} will follow from Lemma \ref{lemma step 1} because we can apply this Lemma for the frame obtained from the original quasi-normal frame by an appropriate permutation.}   
\end{rk}

To show \eqref{target 2}, first, given $i_0\in [1:k]$  construct a submatrix $N_i$ of $A$ with row indices
\begin{equation}
\label{Tdef}
  T_{i_0}:= \left(\bigcup_{i=1}^k [n_{i-1}+1: n_{i-1}+d_i ]\right) \cup\{m+e_{i_0-1}+1\}.
\end{equation}
Then $N_{i_0}$ has the following form:
\begin{equation}
\label{Ni}
    N_{i_0} = 
     \begin{pmatrix}
            N_{i_0,1} & 0 & 0 \\
            0 & \ddots & 0 \\
            0 & 0 & N_{i_0,k}\\
        \hline
        ~&a^2_{n_{i_0-1}+1}&~
    \end{pmatrix},
\end{equation}
where $N_{i_0, j}$ is of size $d_i \times d_i$, and $a^2_{n_{i_0-1}+1}$ is the $(n_{i_0-1}+1)$st row of the matrix $A^2$ from the second layer of the fundamental algebraic system \eqref{A.phi.B}. 
Moreover 
\begin{equation}
\label{Nij}
 N_{i_0, j}=\begin{cases}
 M_{i_0,j},& \text{if  }j\neq i_0\\
 M_{s,i_0}, & \text {if } j=i_0 \,\,(\text{here }  s\neq i_0),
 \end{cases}   
\end{equation}
where $M_{i,j}$ are as in \eqref{M_i}. Note that by Remark \ref{M_indep_rem} the right hand-side in the second line of \eqref{Nij} is independent of $s\neq i_0$.

Denote by
$b|_{T_{i_0}}$ the subcolumn  of the column $b$ of the fundamental algebraic system  \eqref{A.phi.B} consisting of the same rows as in matrix $N_{i_0}$, i.e., the rows of $b$ indexed by  the set $T_{i_0}$ as in  \eqref{Tdef}.

\begin{Lemma}
\label{lemma p2}
 Assume that 
 \eqref{c^k} holds and  given $i_0\in [1:k]$ 
 the coefficients of all monomials of the form
\begin{equation}
\label{monomial 3}
\begin{split}
~&u_l u_s u_j\left(\prod_{i=1}^ku_{n_{i-1}+d_i+1}\right) \left(\prod_{i=1}^k
   (u_{n_{i-1}+1})^{d_i-1}\right) \quad \text{with }\\ 
 ~&  l\in [n_{i_0-1}+2:n_{i_0-1}+d_{i_0}+1], s\in [1:m]\backslash \mathcal I_{i_0}^1, \text{ and } j\in \mathcal I_{i_0}^2 
\end{split}
\end{equation}
in $\det([N_{i_0}|b|_{T_{i_0}}]$ vanish. 
Then
\begin{equation}
\label{cm_1+1js}
    c^{s}_{lr} = 0, \quad  l,r\in [n_{i_0-1}+2:n_{i_0-1}+d_{i_0}+1], \text{ and } s\in [1:m]\backslash \mathcal I_{i_0}^1.
\end{equation}
\end{Lemma}
\begin{proof} Without loss of generality we can assume that $i_0=1$ and that $s\in \mathcal I_2^1$.  Let $\Tilde b_{T_1}$  
be  the subcolumn  of the column $\Tilde b$ consisting of the rows of $\Tilde b$ indexed by  the set $T_{1}$, where $\Tilde b$ is as in \eqref{b^{1,2}_i}.

Further,  using \eqref{Nij},
similar to  the proof of Lemma \ref{lemma step 1} (statements of  items ($A1_i$) and $(A2_i)$ there), one can conclude that  
the variable $u_{n_{i-1}+1}$, $i\in [1:k]$ appears in the following  columns of the augmented matrix $[N_1|\Tilde b|_{T_{1}}]$ only:
\begin{enumerate}
\item[$(\mathrm{F}1_i)$] The columns containing all columns of $N_{1,i}$ except the last one. Moreover,  in each of these columns $u_{n_{i-1}+1}$ appears exactly in the entry $(N_{1, i})_{j+1 j}$, i.e., in the entry situated right below the diagonal of the block $N_{1, i}$. Besides, the coefficient of $u_{n_{i-1}+1}$ in this entry is equal to $1$;
\item[$(\mathrm{F}2_i)$] the last column of the matrix $[ N_1|\Tilde b|_{T_{1}}]$;
\item[$(\mathrm{F}3_i)$] the last row of  the matrix $[ N_1|\Tilde b|_{T_{1}}]$.
\end{enumerate}
Besides,  applying the   normalization conditions \eqref{quasi_normal_struct} to   \eqref{b^{1,2}_i} one gets
\begin{enumerate}
\item [(F$4_i)$] 
The components  of the column vector $\Tilde b|_{T_{1}}$  in the rows corresponding to the rows  of the block $N_{1, i}$ of $N_1$ do not contain $u_{n_{i-1}+1}$.
\end{enumerate}
The following sublemma is important in the sequel, but its proof consists of tedious computations and is postponed to Appendix \ref{appendix B}:

\begin{Sublemma}
\label{last_last_coro}
The entry from the last column and the last row of $[N_1|
\Tilde b|_{T_1}]$, i.e., $\Tilde b_{m+1}$, is independent of $u_j$ with $j\in \mathcal I_1^2$.
\end{Sublemma}
 
From Sublemma \ref{last_last_coro} and the fact that the elements of the blocks $N_{1,v}$ defined by \eqref{Ni} do not depend on variables $u_j$ with $j>m$, it follows that a term $\mathcal L$ in the Leibniz formula for $\det [N_1|\Tilde b|_{T_{1}}]$, which contributes to the monomial \eqref{monomial 3}, cannot contain as a factor the entry from the last column and the last row of the augmented matrix $[N_1|\Tilde b|_{T_{1}}]$. Therefore, the term $\mathcal L$ must contain one factor of the form $\Tilde b_r$, where $r\in\displaystyle{\bigcup_{i=1}^k[n_{i-1}+2:n_{i-1}+d_i]}$ (here we also use Corollary \ref{first_bi_coro}) and one factor from the last row of  $[N_1|\Tilde b|_{T_{1}}]$, i.e., $a^2_{1, m+t}$. Moreover, by Remark \ref{b>m_rem} and since the second line of \eqref{a^2_{1, m+j}_0} is independent of  $u_j$ with $j>m$ , we must have 
\begin{equation}
 \label{t_in}
 t\in [1:d_1].
\end{equation}

Now, for definiteness, assume that $r\in [n_{i_1-1}+2: n_{i_1-1}+d_{i_1}]$ for some $i_1\in [1:k]$. By property ($F4_i$), $\Tilde b_r$ does not contain $u_{n_{i_1-1}+1}$. Besides, the entry $(N_{1, i_1}){r-n_{i_1-1}, r-n_{i_1-1}-1}$ cannot be a factor in $\mathcal L$.  So, using properties 
(F$1_i$)-(F$3_i$) and the fact that 
$u_{n_{i_1-1}+1}$ appears in monomial \eqref{monomial 3} in the power not less than $d_{i_1}-1$ we get that the factor $a^2_{1,m+t}$ must contain $u_{n_{i_1-1}+1}$. Hence, the desired contribution of $a^2_{1,m+t}$ to the monomial \eqref{monomial 3} is equal to the coefficient of the monomial $u_{n_{i_1-1}+1} u_j$ in $a^2_{1,m+t}$.

To find this coefficient note that  $a^{2}_{1, m+d_1}$ is given by \eqref{a^2_{1, m+j}_0}. From \eqref{qjk} it follows that the second term of \eqref{a^2_{1, m+j}_0} depends only on $u_i$'s with $i\in[1:m]$, so it does not contribute to the required monomial. Therefore, we have to find the contribution of the first term, i.e., of $\vec h_1(q_{1, m+t})$. From \eqref{qjk}, the decomposition of the Tanaka symbol \eqref{Tanaka_decomposition}, and the normalization conditions \eqref{quasi_normal_struct} it follows that 
\begin{equation}
   q_{1, m+t} =-u_{t+1} - \sum_{x=d_1+2}^{n_1} c^{m+d_1}_{1x} u_x.
\end{equation}
Then, using \eqref{huj}
\begin{equation}
\label{h1q1d_1+1}
\begin{split}
 ~&  \vec h_1(q_{1 m+t})= -\Vec{h}_1(u_{t+1} + \sum_{x=d_1+2}^{n_1} c^{m+t}_{1x} u_x )=\\
~&-\bigl( \sum_{v=1}^n q_{t+1 v}u_v+\sum_{x=d_1+2}^{n_1} \sum_{v=1}^n  c^{m+t}_{1x}q_{x v}u_v +\sum_{x=d_1+2}^{n_1} \vec h_1(c^{m+t}_{1x}) u_x\bigr).
\end{split}   
\end{equation}
The last term in \eqref{h1q1d_1+1} depends on $u_w$'s with $w\in[1:m]$ only and does not contribute to the coefficient of the monomial $u_{n_{i_1-1}+1} u_{j}$  with $j\in \mathcal I_1^2$ in  $\vec h_1(q_{1, m+d_1})$.
As $q_{ik}$  depends on $u_i$'s with $i\in[1:m]$ only (see \eqref{qjk}),  in the first two terms of \eqref{h1q1d_1+1} only summands with $v=j$  contribute to the coefficient of  monomial $u_{n_{i_1-1}+1} u_{j}$ in  $\vec h_1(q_{1, m+d_1})$.  So, again by \eqref{qjk}, this  coefficient is equal to
\begin{equation}\label{case1: term1_i1}
    -c_{n_{i_1-1} t+1}^{j}-\sum_{x=d_1+2}^{n_1} c^{j}_{n_{i_1-1}x} c^{m+t}_{1x}.
\end{equation}
Since $j\in \mathcal I_1^2$ , by the decomposition of the Tanaka symbol \eqref{Tanaka_decomposition}, the expression in \eqref{case1: term1_i1} is equal to zero for $i_1\in [2:k]$. So a nonzero contribution is obtained only if 
\begin{equation}
\label{i1}
  i_1=1 \Rightarrow r\in[2:d_1].  
\end{equation}
 In this case the coefficient of  monomial $u_1 u_{j}$ in  $\vec h_1(q_{1, m+d_1})$ is equal to
\begin{equation}\label{case1: term1_1}
    -c_{1 t+1}^{j}-\sum_{x=d_1+2}^{n_1} c^{j}_{1x} c^{m+t}_{1x}= -\sum_{x=2}^{n_1} 
 c^{j}_{1x} c^{m+t}_{1x},
\end{equation}
where the last equality follows from 
the normalization conditions \eqref{quasi_normal_struct}.  
If we set
\begin{equation}
\label{vs}
    v^s = 
        (c^s_{1,2},
        c^s_{1,3},
        \dots,
        c^s_{1,m_1})
    \quad s \in [m+1: m+d_1],
\end{equation}
then \eqref{case1: term1_1}  can be rewritten in terms of the standard inner product $\langle\cdot, \cdot\rangle$ in $\mathbb R^{d_1}$ as follows:
\begin{equation}\label{case1: term1}
-\langle v^j, v^{m+t}\rangle.
\end{equation}

Further, as before, for our purpose, it is enough to set variables $u_i$ not appearing in \eqref{monomial 3} to be equal to zero. Let us check what happens in the rows of $[N_1|\Tilde b|_{T_{1}}]$ containing the first rows of each block $N_{1,i}$, Consider the cases $i>2$ and $i\in \{1,2\}$ separately:

{\bf (G1) The case $i>2$}. \emph{Any  term  $\mathcal{L}$ in the Leibniz formula for $\det [N_1|\Tilde b|_{T_{1}}]$ contributing to the monomial \eqref{monomial 3} must contain as factors all entries of the form $(N_{1,i})_{x+1x}$, $x\in [1:d_i-1]$ and also the entry $(N_{1,i})_{1,d_i}$}.
Indeed, in this case, the only nonzero entry in the first row of $N_{1, i}$ is $(N_{1, i})_{1,d_i}$ and, using Corollary \ref{first_bi_coro}, $(N_{1, i})_{1,d_i}$ must appear as a factor in the term $\mathcal L$. Furthermore, if for some $x_0\in [1:d_i-1]$ the entry 
$(N_{1,i})_{x_0+1x_0}$ is omitted, then by \eqref{Ni} and \eqref{t_in} the considered term $\mathcal L$ will contain a factor of the form $(N_{1,i})_{y,x_0}$ for some $y\in[2:d_i]\backslash\{x_0+1\}$. Hence, also the entry $(N_{1, i})_{y y-1}$ does not appear as a factor in the considered term $\mathcal L$. Then from properties (F$1_i$)-(F$3_i$)
and the fact that by \eqref{i1} the factor $a^2_{1 m+t}$ from the last row of $N_1$ does not contain $u_{n_{i-1}+1}$, it follows that the power of $u_{n_{i-1}+1}$ in the considered term is less than $d_i-1$, but in the monomial \eqref{monomial 3} it is at least $d_i-1$ and we got the contradiction. Note also that in the considered case, using the normalization conditions \eqref{quasi_normal_struct}, we have 
\begin{equation}
\label{N1>2}
(N_{1,i})_{x+1x}=u_{n_{i-1}+1}, \quad (N_{1,i})_{1,d_i}=-u_{n_{i-1}+d_i+1}, 
\end{equation}
so the total contribution of these factors to the coefficient of the monomial \eqref{monomial 3} is $\pm 1$, i.e., trivial.

{\bf (G2) The cases $i=1$ and $i=2$}. From Corollary \ref{first_bi_coro} it follows that 
\begin{itemize}
\item the only nonzero entries in the first row of  $[N_1|\Tilde b|_{T_{1}}]$ are
\begin{equation}
\label{N11}
(N_{1,1})_{1\, l-1}=-u_l, \quad (N_{1,1})_{1\,d_1}=-u_{d_1+1};
\end{equation}

\item the only nonzero entries in the row of  $[N_1|\Tilde b|_{T_{1}}]$ containing the first row of $N_{1,2}$ are 
\begin{eqnarray}
~&(N_{1,2})_{1\, s-n_1-1}=-u_s, \quad \textrm{if } s\in [n_1+2: n_1+d_1+1]
\label{N12s}\\
~&
(N_{1,2})_{1\,d_1}=-u_{n_1+d_2+1}.\label{N12d1}
\end{eqnarray}
\end{itemize}
Consequently, we have the following four possibilities:

{\bf (G2a) The entries $(N_{1,1})_{1\,d_1}$ and 
$(N_{1,2})_{1\,d_2}$ are factors in $\mathcal L$.}
This  and \eqref{i1} implies that the entries 
$(N_{1,2})_{x+1, x}$  for all $x\in [1:d_2-1]$ are factors in $\mathcal L$, and there exists $r\in[2:d_1]$ such that $\Tilde b_r$ is a factor in $\mathcal L$. The latter implies that the entry $(N_{1,1})_{r, r-1}$ is not a factor in $\mathcal L$ and therefore $a^2_{1, m+r-1}$ is a factor in $\mathcal L$, otherwise  more than one of the entries of the form 
$(N_{1,1})_{x+1, x}$, $x\in [1: d_1-1]$ will not appear as factors in $N_{1,1}$ and it will contradicts the fact that $u_1$ appear in monomial \eqref{monomial 3} in the power not less than $d_1-1$.

Recall that in $a^2_{1, m+r-1}$ we are interested in the coefficient of the monomial $u_1 u_j$ which by \eqref{case1: term1} is equal to 
\begin{equation}\label{case1: term1_r}
-\langle v^j, v^{m+r-1}\rangle.
\end{equation}
Taking into account all factors we revealed in $\mathcal L$, relations \eqref{N1>2}, \eqref{N11}, \eqref{N12d1}, and \eqref{N1>2}, and also that in $a^2_{1, m+r-1}$ we are interested in the monomial $u_1 u_j$, we conclude that in $\Tilde b_r$ we are interested in the coefficient of $u_l u_s$, as $u_l$ and $u_s$ ar ethe only variables in monomial \eqref{monomial 3} that still were not used. By \eqref{b^{1,2}_i} the latter coefficient is equal to $(\alpha_1^2-\alpha_2^2)c_{lr}^s$. So, this together with \eqref{case1: term1_r} and the last sentence of (G1) implies the total contribution of all possible terms satisfying the assumptions of (G2a) (i.e., for every $r\in [2, d_1]$ can be written as
\begin{equation}
\label{G2a_con}
(\alpha_1^2-\alpha_2^2)\langle \sum_{r=2}^{d_2} \mathrm {sgn}(\sigma_r) c_{lr}^s v^{m+r-1}, v^j\rangle,
\end{equation}
where $\mathrm{sgn}(\sigma_r)$ stands for the sign of the permutation related to the corresponding term in the Leibniz formula for $\det [N_1|\Tilde b|_{T_{1}}]$. Since the value of $\mathrm{sgn}(\sigma_r)$ is not important to the final conclusions, its explicit expression is not written out here. 

{\bf (G2b) The entries $(N_{1,1})_{1\,l-1}$ and 
$(N_{1,2})_{1\,d_2}$ are factors in $\mathcal L$}.
First, regarding the entries from the block $N_{1,2}$ we can use the same arguments and the same conclusions as in (G1a). Regarding the entries of the block $N_{1,1}$, since the entry $(N_{1,1})_{1\,l-1}$ is already used in $\mathcal L$, the entry  $(N_{1,1})_{l\,l-1}$ should not appear as a factor in $\mathcal L$, but then, by the same arguments about the lower bound for the power of $u_1$ that we already used both in (G1) and (G2a),  all other entries of the form $(N_{1,1})_{x, x-1}$ has to be in $\mathcal L$. Therefore, the entries $a^2_{1, m+d_1}$ and $\Tilde b_l$ must appear as factors   of $\mathcal L$. By \eqref{case1: term1} the former contributes to monomial \eqref{monomial 3}  the factor 
\begin{equation}\label{case1: term1_d_1}
-\langle v^j, v^{m+d_1}\rangle.
\end{equation}
Next, similarly to (G2a), the contribution to monomial \eqref{monomial 3} from $\Tilde b_l$ equal to the coefficient of $u_{d_1+1} u_s$ which by \eqref{b^{1,2}_i}is  equal to $-(\alpha_1^2-\alpha_2^2)c_{l d_1+1}^s$ so that the contribution of the terms from (G2b) to the monomial \eqref{monomial 3} is equal to
\begin{equation}
 \label{G2b_con} (\alpha_1^2-\alpha_2^2)\langle  \mathrm {sgn}(\sigma_{d_2+1}) c_{ld_2+1}^s v^{m+d_2}, v^j\rangle,
\end{equation}
where ${sgn}(\sigma_{d_2+1})$ stands for the sign of the corresponding permutation.
Combining \eqref{G2a_con} and \eqref{G2b_con} we get that the contribution of the terms satisfying conditions of (G2a) or (G2b) can be written as
\begin{equation}
\label{G2a&b_con}
(\alpha_1^2-\alpha_2^2)\langle \sum_{r=2}^{d_2+1} \mathrm {sgn}(\sigma_r) c_{lr}^s v^{m+r-1}, v^j\rangle.
\end{equation}

Further note that by \eqref{N12s} for  $s\notin [n_1+2: n_1+d_1+1]$ the cases (G2a) and G2b) are the only possible cases, so in this case the expression in \eqref{G2a&b_con} is equal to the coefficient of the monomial \eqref{monomial 3}. Vanishing of this coefficient for any $j\in \mathcal I_1^2$ implies that the vector $ \sum_{r=2}^{d_1+1} \mathrm{sgn}(\sigma_r) c_{lr}^s v^{m+r-1}$, which  belongs to $\mathrm{span} \{v^{x}:x\in \mathcal I_1^2 \}$  is orthogonal to  $v^{j}$ for all $j\in \mathcal I_1^2$.
Hence, 
\begin{equation}
\label{inner prod is 0}
\sum_{r=2}^{d_1+1} \mathrm{sgn}(\sigma_r) c_{lr}^s v^{m+r-1} = 0.
\end{equation}
From $\mathrm{ad}$-surjectivity, and, more precisely, since $X_1$ is chosen as $\mathrm{ad}$-generating element, the tuple of vectors $(v^{m+1}, v^{m+2}, \ldots, v^{m+d_1})$ form a basis in $\mathbb R^{d^1}$. Therefore, 
\begin{equation}
\label{cm_1+1js_inter}
    c^{s}_{lr} = 0, \text{ for } l,r\in [2:d_1+1], \text{ and } s \notin [1: d_1]\cup [n_1+2: n_1+d_1+1]. 
\end{equation}

{\bf (G2c) The entries $(N_{1,1})_{1\,d_1}$ and 
$(N_{1,2})_{1\,s-n_1-1}$ are factors in $\mathcal L$ (relevant only for $s\in [n_1+2:n_1+d_2+1]$).} For the block $N_{1,2}$ let us apply the  arguments similar to the ones we used for the block $N_{1,1}$ in (G2b):

Since the entry $(N_{1,2})_{1\,s-n_1-1}$ is already used in $\mathcal L$, the entry  $(N_{1,2})_{s-n_1\,s-n_1-1}$ should not appear as a factor in $\mathcal L$, but then, by the same arguments about the lower bound for the power of $u_1$ that we already used both in (G1) and (G2a),  all other entries of the form $(N_{1,2})_{x, x-1}$, $x\in [2:d_2]$ has to be in $\mathcal L$. Then the contribution ro $\mathcal L$ of the row of the matrix $[N_1|\Tilde b|_{T_{1}}]$ containing the $(s-n_1)$st row 
of $N_{1,2}$ must be equal to the entry $(N_{1,2})_{s-n_1, d_2}$. Here we also use that this contribution cannot be from the last column of $[N_1|\Tilde b|_{T_{1}}]$, because of \eqref{i1}. Since by \eqref{N12s}
$U_s$ is already used, the term in  $(N_{1,2})_{s-n_1, d_2}$ that may give a contribution to monomial \eqref{monomial 3} is the one containing $u_{n_1+d_2+1}$. Further, passing to the rows of $[N_1|\Tilde b|_{T_{1}}]$ containing intersecting the block $N_{1,1}$, since $u_{d_1+1}$ is used in $(N_{1,1})_{d_1}$, the only possible  contribution  of $\Tilde b_r$ comes from the monomial $u_l u_{n_1+1}$, coefficient of which by \eqref{b^{1,2}_i} is equal to $(\alpha_1^2-\alpha_2^2)c_{l,r}^{n_1+1}$. However, this coefficient is equal to zero by \eqref{cm_1+1js_inter}. Hence, there is no contribution to monomial \eqref{monomial 3}  from terms satisfying (G2c).

{\bf (G2d) The entries $(N_{1,1})_{1\,l-1}$ and 
$(N_{1,2})_{1\,s-n_1-1}$ are factors in $\mathcal L$ (relevant only for $s\in [n_1+2:n_1+d_2+1]$).} The only difference with (G2c) is that since $(N_{1,1})_{1\,l-1}$ in $\mathcal L$  is used, then one has to use the term $\Tilde b_l$, and since, by \eqref{N11}, is alredy used in $(N_{1,1})_{1\,l-1}$, in $\Tilde b_l$ we are interested in the coefficient of the monomial $u_{d_1+1} u_{n_1+1}$. By \eqref{b^{1,2}_i}, 
this coefficient is equal to $(\alpha_1^-\alpha_2^2)c_{d_1+1 l}^{n_1+1}$, which is equal to zero by by \eqref{cm_1+1js_inter}. Hence, there is no contribution to monomial \eqref{monomial 3}  from terms satisfying (G2d).

So, also for $s\in [n_1+2: n_1+d_1+1]$ the coefficient of the monomial \eqref{monomial 3} in $\det[N_1|\Tilde b|_{T_{1}}]$ is equal to  \eqref{G2a&b_con} Repeating the same arguments as after \eqref{G2a&b_con},  we will get that the equality in  \eqref{cm_1+1js_inter} holds also for $s\in [n_1+2: n_1+d_1+1]$.
This completes the proof of the  Lemma \ref{lemma p2} in the case of $i_0=1$ and $s\in \mathcal I_2^1$. Since we always can permute the indices, it also proves this lemma in general.
\end{proof}

Lemma \ref{lemma p2} 
proves only a subset of relations from \eqref{target 2}.   Using again the flexibility given by Corollary \ref{indep_cor} one can show that \eqref{target 2} holds not for the original quasi-normal frame but for its perturbation adapted to the same choices of $X_{n_{i-1}+1}$, which will be enough to finish the proof of Proposition \ref{step1_lemma2}.

\begin{Coro}
\label{YZ_cor}
Let $K_{X_{n_{i-1}+1}}$ and $H_{X_{n_{i-1}+1}}$ be as in \eqref{K_def} and \eqref{H_def}, respectively. If $Y_i$ and $Z_i$ are sections of $H_{X_{n_{i-1}+1}}$ satisfying
\begin{equation}
\label{YZK}
\dim\mathrm{span}\{Y_i, Z_i\}=2, \quad \mathrm{span}\{Y_i, Z_i\}\cap K_{X_{n_{i-1}+1}}=0,
\end{equation}
then 
\begin{equation}
\label{[Y, Z]}
    [Y_i, Z_i]\in L_{X_{n_{i-1}+1}}\quad \mathrm{mod}D_i,
\end{equation}
where $\mathcal L_{X_{n_{i-1}+1}}$ is defined as in \eqref{Li}.
\end{Coro}
\begin{proof}
Indeed, from \eqref{YZK} it follows that we can find the quasi-normal frame $(X_1, \ldots, X_n)$ such that $X_{n_{i-1}+2}=\cfrac{Y_i}{(g_1(Y_i,Y_i))^{1/2}}$ and $X_{n_{i-1}+3}\in \mathrm{span}\{Y_i,Z_i\}$. Then \eqref{[Y, Z]} follows from Lemmas \ref{lemma p2} and Corollary \ref{indep_cor}.
\end{proof}

By Remark \ref{d11_rem}, we can assume that $d_i\geq 2$, as for $d_i=0$ or $1$, relations \eqref{target 2} are already proved there maybe for a perturbed quasinormal frame. 
Under this assumption the set of planes in $H_{X_{{n_{i-1}+1}}}$ having a trivial intersection with $K_{X_{n_{i-1}+1}}$ is generic. Therefore by a finite number of consecutive small perturbations, one can build a quasi-normal frame  $(X_1, \ldots, X_n)$  such that for every ${n_{i-1}+2}\leq i<j\leq n_i$, the pair $(Y_i, Z_i)=(X_i, X_j)$ satisfies \eqref{YZK} and 
so by Corollary \ref{YZ_cor} this frame satisfies \eqref{target 2}. Besides, by Remark \ref{pert_indep_remark},
the new quasi-normal frame obtained in this step will automatically satisfy
\eqref{struct_3_distinct_reform_2}.
The proof of Proposition \ref{step1_lemma2} is completed.

\subsection {Proof of Proposition \ref{PropD2D3}}
\label{D2D3_sec}

From \eqref{target 2} and Lemma \ref{b^1_the same_1}, using \eqref{qjk}, one can get that   
\begin{equation}
\label{qjk_first_group}
\Tilde b_j=0, \quad \forall j\in [1:m].
\end{equation}
This together with \eqref{column_op} implies that 
if for all $v\in [1:k]$ and $j\in I_v^2$ we set $\Psi_j:=\Phi_j-\alpha_{n_{v-1}+1}^2 u_j$ and $\Psi:=(\Psi_1, \ldots \Psi_m)^T$
, then 
$A_1 \Psi=0$.
This, again together with \eqref{column_op}, implies that the tuple $\widetilde \Phi= (\Phi_{m+1},\ldots \Phi_n)^T$ with
\begin{equation}
\label{Phi_triv}
    \Phi_j = \alpha_{n_{v-1}+1}^2 u_j, \quad j\in \mathcal I_v^2.\\
\end{equation}
which corresponds to $\Psi=0$, is the solution to the first layer $A^1\widetilde\Phi=b^1$ of the Fundamental Algebraic System \eqref{A.phi.B}. 

Further, note that the $m\times (n-m)$-matrix $A^1$ has the maximal rank $n-m$ at a generic point, as from the normalization conditions \eqref{quasi_normal_struct} the coefficient of the monomial $\displaystyle{\prod_{v=1}^k u_{n_{v-1}+1}^{d_v}}$ in its maximal minor consisting of rows from the set  $\displaystyle{\bigcup_{v=1}^k[n_{v-1}+2:n_{v-1}+d_i+1]}$  is equal to $1$. Hence, \eqref{Phi_triv} defines the unique (rational in $u$'s) solution of the 
the system $A^1\widetilde\Phi=b^1$ and therefore it must coincide with the solution of the whole Fundamental Algebraic System \eqref{A.phi.B}, i.e., must satisfy other layers of it. 
In other words,  
\begin{equation}
\label{all_layer}
\sum_{v=1}^k\sum_{l\in \mathcal I_v^2} a^{s}_{t,l} \alpha_{n_{i-1}+1}^2 u_l  = b^{s}_t, \quad \forall s \geq 1, t\in [1:m],
\end{equation}
where  $a^{s}_{t,k}$ and  $b^{s}_t$ satisfy \eqref{elem.A} and \eqref{term.of.b.rec.form}, respectively.
Note that by \eqref{column_op}, the relation \eqref{all_layer} is equivalent to 
\begin{equation}\label{2nd layer monomial}
    \Tilde b=0.
\end{equation}

Without loss of generality, it is enough to prove  Proposition \ref{PropD2D3} for $i=1$.
\begin{Lemma}
\label{4.5_lem_1}
The coefficients of the monomials
\begin{equation}\label{monomial 5}
    u_{y}u_{l}u_{j}, \quad y,l\in [2:d_1+1], j\in \mathcal I_v^1\cup
    \mathcal I_v^2, v\neq 1
\end{equation}
in $\Tilde b_{m+1}$ are equal to
\begin{equation}
\label{sec4.5_1}
\begin{cases}
   - \left(\alpha_1^2 -\alpha_{n_{v-1}+1}^2 \right) \left( c^{j}_{y,m+l-1} + c^{j}_{l, m+y-1} \right), & \quad y\neq l \in [2:d_1+1], j\in 
   \mathcal I_v^1\cup
   \mathcal I_v^2, v\neq 1
    \\
    -\left(\alpha_1^2 -\alpha_{n_{v-1}+1}^2 \right) c^{j}_{y,m+y-1}, & \quad y=l \in [2:d_1+1], j\in 
    \mathcal I_v^1\cup
    \mathcal I_v^2, v\neq 1. 
\end{cases}
\end{equation}
\end{Lemma}
\begin{proof}
 We use the expression for  $\Tilde b_{m+1}$ given by \eqref{b1,2 step 2_fin} in Appendix B. Consider the cases $j\in \mathcal I_v^2$ and $j\in \mathcal I_v^1$ separately.

 {\bf Case 1: $j\in \mathcal I_v^2$.} The first term of this expression does not contribute to the monomial \eqref{monomial 5}, because every monomial in it contains $u_w$ with $w\in [1:m]\backslash \mathcal I_1^1$.
Also, since $w\in [1:m]\backslash \mathcal I_1^1$, by Lemma \ref{lemma step 1} $q_{1,w}$ are independent of $u_z$ for $z\in \mathcal I_1^1$, hence the second term of \eqref{b1,2 step 2_fin} does not contain monomial \eqref{monomial 5}. 
Further,  the third term of \eqref{b1,2 step 2_fin} does not depend on $u_j$ with $j>m$ because by \eqref{qjk} all $q_{jl}$ depend on $u_i$'s with $i\in [1:m]$ only.
 Therefore, in the considered case  the monomial \eqref{monomial 5} may appear in \eqref{b1,2 step 2_fin} only in the fourth term, i.e., in   
\begin{equation}
\label{4th_term}
\sum _{i=2}^k\left( (\alpha_{1})^2 -(\alpha_{n_{i-1}+1})^2\right)\sum_{r = m+1}^{m+d_1}\sum_{t\in \mathcal I_i^2}q_{1r} q_{rt} u_t.
\end{equation}
Finally, by the normalization conditions \eqref{quasi_normal_struct} one can easily show that the coefficient of this monomial in \eqref{4th_term} is equal to \eqref{sec4.5_1}.
 
 {\bf Case 2: $j\in \mathcal I_v^1$.} In this case the fourth term of \eqref{b1,2 step 2_fin} does not contribute to monomial 
\eqref{monomial 5} in $\Tilde b_{m+1}$, because each monomial in this term contains factor $u_t$ with $t>m$. 

Let us analyze the second term of \eqref{b1,2 step 2_fin}: by Lemma \ref{lemma step 1} the factor $q_{1w}$ does not contain $u_s$ with $s\in \mathcal I_1^1$, so it must contribute the coefficient of $u_j$, which is equal to $c_{j1}^w$. Hence, the index $x$, appearing in the second term of  \eqref{b1,2 step 2_fin} must be in $I_{1}^1$, hence must be equal to either $y$ or $l$, where $y$ and $l$ are as in \eqref{monomial 5}. So, the total contribution of the second term of \eqref{b1,2 step 2_fin} to the coefficient of monomial \eqref{monomial 5}
in $\Tilde b_{m+1}$ is equal to 
\begin{equation}
\label{4.5.1_second_term}
\sum _{i=2}^k\left( (\alpha_{1})^2 -(\alpha_{n_{i-1}+1})^2\right)\sum_{w\in\mathcal I_i^1} 
    c_{j1}^w(c_{lw}^y+c_{yw}^l)=0,
\end{equation}
because $c_{lw}^y+c_{yw}^l=0$ by item (2) of Proposition \ref{first divi}. So, the second  term of \eqref{b1,2 step 2_fin} does not contribute to monomial 
\eqref{monomial 5} in $\Tilde b_{m+1}$

Now let us analyze the first term of \eqref{b1,2 step 2_fin}: First, the index $w$ appearing there must be equal to $j$ from \eqref{monomial 5}. Second, as in \eqref{hq},
\begin{equation} 
\label{hqj}
\Vec{h}_1(q_{1j})= \sum_{r=1}^m \sum_{x=1}^n c^j_{r1} q_{rx}u_x + \sum_{r=1}^m \Vec{h}_1(c^j_{r1})u_r.
\end{equation}

In the second term of \eqref{hqj} we are interested in $r\in \mathcal I_1^1$, but in this case, by Lemma \ref{lemma step 1}, $c_{r1}^j=0$ and therefore the second term of \eqref{hqj} does not contribute to \eqref{monomial 5_0}. 

By the same reason, the index $r$ in the first term of \eqref{hqj} can  be taken from $[1:m]\backslash I_1^1$, Hence, the coefficient of $u_y u_l$ we are interested in the first term of \eqref{hqj} for the monomial \eqref{monomial 5} is equal to
\begin{equation}
\label{hqj_first_term}
\sum_{r=m_1+1}^m
c^j_{r1}(c_{lr}^y+c_{yr}^l)=0,
\end{equation}
because $c_{lr}^y+c_{yr}^l=0$  by item (2) of Proposition \ref{first divi} again. Consequently, the first  term of \eqref{b1,2 step 2_fin} does not contribute to monomial 
\eqref{monomial 5} in $\Tilde b_{m+1}$.

So, in the considered case  the monomial \eqref{monomial 5} may appear in \eqref{b1,2 step 2_fin} only in the third term, i.e., in   
\begin{equation}
\label{3th_term}
\sum _{i=2}^k\left( (\alpha_{1})^2 -(\alpha_{n_{i-1}+1})^2\right)\sum_{r = m+1}^{m+d_1}  \sum_{w\in \mathcal I_i^1}   q_{1r} q_{rw} u_w.
\end{equation}
Finally, by the normalization conditions \eqref{quasi_normal_struct} one can easily show that the coefficient of this monomial in \eqref{3th_term} is equal to \eqref{sec4.5_1}.

\end{proof}

From  condition \eqref{2nd layer monomial} and Lemma \ref{4.5_lem_2} it follows that 
\begin{equation}
\label{sec4.5_1=0}
\begin{cases}
    c^{j}_{y,m+l-1} + c^{j}_{l, m+y-1}=0, & \quad y\neq l \in [2:d_1+1], j\in [1: n]\backslash (\mathcal I_1^1\cup\mathcal I_1^2); 
    \\
    c^{j}_{y,m+y-1}=0, & \quad y=l \in [2:d_1+1], j\in [1: n]\backslash (\mathcal I_1^1\cup\mathcal I_1^2)  .
\end{cases}
\end{equation}
\begin{Lemma}
\label{4.5_lem_2}
The coefficients of the monomials
\begin{equation}\label{monomial 5_0}
    u_{1}^2u_{j}, \quad, j\in \mathcal I_v^1\cup
    \mathcal I_v^2, v\neq 1
\end{equation}
in $\Tilde b_{m+f}$, $f\in [2:d_1+1]$ are equal to
\begin{equation}
\label{sec4.5_0}
   \left(\alpha_1^2 -\alpha_{n_{v-1}+1}^2 \right) c^{j}_{1,m+f-1}. 
   \end{equation}
\end{Lemma}
\begin{proof}
 We use the expression for  $\Tilde b_{m+f}$ given by formula \eqref{b1,2 step 2_fin} from Appendix B.  Consider the cases $j\in \mathcal I_v^2$ and $j\in \mathcal I_v^1$ separately.

 {\bf Case 1: $j\in \mathcal I_v^2$.}
 The first and the third term of  \eqref{b1,2 step 2_fin} do not contribute to the monomial \eqref{monomial 5_0} by the same arguments as in the proof of case 1 of Lemma \ref{4.5_lem_1}, while the same holds for the second term not only by Lemma \ref{lemma step 1} but also by Lemma \ref{lemma p2}. Hence, as in the proof of Lemma \ref{4.5_lem_1}, the monomial \eqref{monomial 5_0} may appear in the fourth term only, i.e., in 
 \begin{equation}
\label{4th_term_f}
\sum _{i=2}^k\left( (\alpha_{1})^2 -(\alpha_{n_{i-1}+1})^2\right)\sum_{r = m+1}^{m+d_1}\sum_{t\in \mathcal I_i^2}q_{fr} q_{rt} u_t.
\end{equation}
Finally, by the normalization conditions \eqref{quasi_normal_struct} one can easily show that the coefficient of this monomial in \eqref{4th_term_f} is equal to \eqref{sec4.5_0}.

 {\bf Case 2: $j\in \mathcal I_v^1$.} In this case the fourth term of \eqref{b1,2 step 2_fin} does not contribute to monomial 
\eqref{monomial 5_0} in $\Tilde b_{m+1}$ by the same reason as in the proof of case 2 of Lemma \ref{4.5_lem_1}.

The analysis of the first and second terms of \eqref{b1,2 step 2_fin} is also completely analogous to the one in the proof of case 2 of Lemma \ref{4.5_lem_1}. The main differences are that here we use Lemma \ref{lemma p2} together with  Lemma \ref{lemma step 1} there and item (1) of Proposition \ref{first divi}  instead of item (2) of Proposition \ref{first divi} there.

In more detail, for the second  term of \eqref{b1,2 step 2_fin}
by Lemma \ref{lemma p2} the factor $q_{fw}$ does not contain $u_s$ with $s\in \mathcal {I}_1^1$, so it must contribute the coefficient $u_j$, which is equal to $c_{j1}^w$. Hence the index $x$, appearing in the second term of  \eqref{b1,2 step 2_fin} must be in $I_{1}^1$, hence must be equal to $1$. So, the total contribution of the second term of \eqref{b1,2 step 2_fin} to the coefficient of monomial \eqref{monomial 5_0}
in $\Tilde b_{m+f}$ is equal to 
\begin{equation}
\label{4.5.2_second_term}
\sum _{i=2}^k\left( (\alpha_{1})^2 -(\alpha_{n_{i-1}+1})^2\right)\sum_{w\in\mathcal I_i^1} 
    c_{jf}^wc_{1w}^1=0,
\end{equation}
because $c_{1w}^1=0$ by item (2) of Proposition \ref{first divi}. So, the second  term of \eqref{b1,2 step 2_fin} does not contribute to monomial 
\eqref{monomial 5_0} in $\Tilde b_{m+f}$.

Now let us analyze the first term of \eqref{b1,2 step 2_fin}: First, the index $w$ appearing there must be equal to $j$ from \eqref{monomial 5_0}. Second, as in \eqref{hq},
\begin{equation} 
\label{hqj_f}
\Vec{h}_1(q_{fj})= \sum_{r=1}^m \sum_{x=1}^n c^j_{rf} q_{rx}u_x + \sum_{r=1}^m \Vec{h}_1(c^j_{rf})u_r.
\end{equation}

In the second term of \eqref{hqj_first_term} we are interested in $r=1$, but in this case, by Lemma \ref{lemma step 1}, $c_{r1}^j=0$ and therefore the second term of \eqref{hqj} does not contribute to \eqref{monomial 5_0}. 

By similar reasoning,  using now also Lemma \ref{lemma p2}, the index $r$ in the first term of \eqref{hqj} can  be taken from $[1:m]\backslash I_1^1$, Hence, the coefficient of $u_1^2$, we are interested in the first term of \eqref{hqj_f} for the monomial \eqref{monomial 5_0} is equal to
\begin{equation}
\label{hqj_first_term_1}
\sum_{r=m_1+1}^m
c^j_{rf}c_{1r}^r=0,
\end{equation}
because $c_{1r}^1=0$  by item (1) of Proposition \ref{first divi} again. Consequently, the first  term of \eqref{b1,2 step 2_fin} does not contribute to monomial 
\eqref{monomial 5_0} in $\Tilde b_{m+f}$.

So, in the considered case the monomial \eqref{monomial 5} may appear in \eqref{b1,2 step 2_fin} only in the third term, i.e., in   \eqref{3th_term}. Finally, by the normalization conditions \eqref{quasi_normal_struct} one can easily show that the coefficient of this monomial in \eqref{3th_term} is equal to \eqref{sec4.5_0}.

\end{proof}

From  condition \eqref{2nd layer monomial} and Lemma \ref{4.5_lem_2} it follows that 
\begin{equation}
\label{adX1_1_struct} 
c^{j}_{1,l}=0, \quad l\in \mathcal I_1^2, j\in [1: n]\backslash (\mathcal I_1^1\cup\mathcal I_1^2). 
\end{equation}
In other words, 
\begin{equation}
\label{adX1_1} 
\mathrm{ad}(X_1) (D_1^2) \subset D_1^2.
\end{equation}


\begin{Lemma}
\label{Jacobi_lem1}
The following inclusion holds
\begin{equation}
\label{adXy} 
\mathrm{ad}(X_y) (D_1^2) \subset D_1^2,\quad \forall y\in \mathcal I_1^1.
\end{equation}
\end{Lemma}
\begin{proof}

Take  $y, l \in [2:d_1+1]$.  Using the normalization conditions \eqref{quasi_complete} and the  Jacobi identity, we have 
\begin{equation}
    \begin{split}
    ~&[X_{y}, X_{m+l-1}] \stackrel{\eqref{quasi_complete}}{=} [X_{y},[X_1,X_{l}]] =
[[X_{y},X_1],X_{l}] +\\~& [X_1, [X_{y}, X_{l}]] 
        \stackrel{\eqref{quasi_complete}}{=} [-X_{m+y-1}, X_{l}] + [X_1, [X_{y}, X_{l}]].
    \end{split}
\end{equation}
Hence, using \eqref{adX1_1}, \begin{equation}\label{2nd layer jac id}
[X_{y},X_{m+l-1}] - [X_{l}, X_{m+y-1}] =  [X_1, [X_{y}, X_{l}]] \in \mathrm{ad}(X_1) (D_1^2) \subset D_1^2,
\end{equation}
which implies
\begin{equation}
    \label{2nd layer step 2}
    c^{j}_{y,m+l-1} - c^{j}_{l, m+y-1} =0,  \quad y, l \in [2:d_1+1], j\in [1: n]\backslash (\mathcal I_1^1\cup\mathcal I_1^2).
\end{equation}
Combining \eqref{sec4.5_1=0} and \eqref{2nd layer step 2}, we get 
\begin{equation}
    c^{j}_{y,m+l-1} = c^{j}_{l, m+y-1} =0,  \quad \quad y, l \in [2:d_1+1], j\in [1: n]\backslash (\mathcal I_1^1\cup\mathcal I_1^2).
\end{equation}
In other words, taking into account \eqref{adX1_1} we get 
\begin{equation}
\label{adXy_inter} 
\mathrm{ad}(X_y) (D_1^2) \subset D_1^2,\quad \forall y\in [1:d_1+1].
\end{equation}
Furthermore, when $y\in [d_1+2:  m_1]$ and $l\in [1:d_1]$, the Jacobi identity yields
\begin{equation}
    \begin{split}
      ~&  [X_y, X_{m+l}] = [X_y, [X_1, X_{l+1}]] 
        = [[X_y, X_1], X_{l+1}] + [X_1, [X_y, X_{l+1}]]\subset\\~&\mathrm{ad}(X_{l+1})(D_1^2) + \mathrm{ad}(X_{1})(D_1^2) 
\stackrel{\eqref{adXy_inter}}{\subset} (D_1^2).
    \end{split},
\end{equation}
which implies \eqref{adXy}.
\end{proof}
Finally, observe that \eqref{adXy} implies 
$(D_1)^3(q) = (D_1)^2(q)$.
 The proof of Proposition \ref{PropD2D3} is completed.

\subsection{Completing the proof of Proposition \ref{Dij^2_involutive}}
\label{Dij_prop_sec}
Without loss of generality, it is enough to prove that $D_1^2+D_2^2$ is involutive. Moreover, for this, by Proposition \ref{PropD2D3}, it is enough to show 
that 
\begin{equation}
\label{4.6_goal}
[X_r, X_j]\in D_1^2+D_2^2,\quad  (r,j)\in (\mathcal I_1^1\times \mathcal I_2^2)\cup (\mathcal I_1^2\times \mathcal I_2^1).
\end{equation} 
Without loss of generality (swapping the indices), it is enough to show it for  $(r,j)\in (\mathcal I_1^2\times \mathcal I_2^1)$. For this, using the normalization conditions \eqref{quasi_complete} and the Jacobi identity, we get
\begin{equation}
\label{4.6_Jac}
[X_r, X_j]\stackrel{\eqref{quasi_complete}}{=}[[X_1, X_{r-m+1}], X_j]=[[X_1, X_j], X_{r-m+1}]+[X_1, [X_{r-m+1}, X_{j}]].
\end{equation}
Since $1$ and $r-m+1$  belong to  $\mathcal I_{1}^1$ and $j\in \mathcal I_2$, by the decomposition of the Tanaka symbol \eqref{Tanaka_decomposition} and Lemma \ref{lemma step 1 k com} it follows that 
\begin{equation}
\label{no_square}
[X_1, X_j] \in D_1+D_2, \quad [X_{r-m+1}, X_{j}] \in D_1+D_2.
\end{equation}
Consequently, using Lemma \ref{lemma step 1 k com}, from the right hand-side of  \eqref{4.6_Jac} and \eqref{no_square} it follows that
$$[X_r, X_j]\in D_1^2+D_2\subset D_1^2+D_2^2,$$
which completes the proof of Proposition \ref{Dij^2_involutive}.

\subsection{Completing the proof of Theorem \ref{mainthm} by rotating the frames on $D_i$'s}\label{existence rotation}
\label{rotation_sec}
By Proposition \ref{PropD2D3}, we have that  the distributions $D_i^2$, $i\in[1:k]$ are involutive. Moreover, by  Proposition \ref{Dij^2_involutive}  the distributions $D_i^2 \oplus D_j^2$, $i,j\in [1:k]$ are involutive as well. The latter implies that for every subset 
$\mathcal {J}\subset [1:k]$ the distribution
\begin{equation} \Delta_{\mathcal J}:=
\bigoplus _{i\in \mathcal{J}} D_i^2
\end{equation}
is involutive. 
Let $\mathcal{F}_{\mathcal J}(q)$ be the integral submanifolds of the distirbution $\Delta_{\mathcal J}$ passing through the point $q$. 

By constructions, for any $i\in [1:k]$ and point $q \in M$ , the manifold  $\mathcal{F}_{\{i\}}(q)$ is transversal to $\mathcal{F}_{[1:k]\backslash\{i\}} (q)$. By the transversality and the Inverse Function Theorem, for any $q_0\in M$ there exists a neighborhood $\mathcal{U}$ of $q_0$ such that for any $q_1\in \mathcal{U}$ and  $q_2\in \mathcal{U}$, the leaf $\mathcal{F}_{\{i\}}(q_1)$ intersects with the leaf $\mathcal{F}_{[1:k]\backslash\{i\}}(q_2)$ at exactly one point. Then the following projection maps $\pi_i^{q_0}: \mathcal{U} \to \mathcal{F}_{\{i\}}(q_0) $ 
is well defined: 
\begin{equation}
       \pi_i^{q_0} (q) = \mathcal{F}_{\{i\}}(q_0) \cap \mathcal{F}_{[1:k]\backslash\{i\}}(q)\\
\end{equation}
for any $q\in \mathcal{U}$. Note that the map $\pi_i^{q_0}$ is a diffeomorphism  between $\mathcal{F}_i(q)\cap U$ 
 and $\mathcal{F}_i(q_0)\cap U$. As a consequence, for any $r \in \mathcal{I}_1^i$, there exist a unique vector field $\Tilde{X}_r$ on $\mathcal{U}$ such that 
\begin{equation}\label{eq:15}
    \begin{split}
    ~& \Tilde{X}_r \textrm{ is a section of } D_i^2 \text { and}\\ 
~&\mathrm{d}\pi_i^{q_0}\bigl(\Tilde{X}_r(q) \bigr)=X_r(\pi_i^{q_0}(q)), \quad \textrm{ where } r \in \mathcal{I}_1^i.
    \end{split}
\end{equation}
\begin{Lemma}
\label{section_lemma}
    $\Tilde{X}_r$ is a section of $D_i$ for every $r\in \mathcal I_i^1$.
\end{Lemma} 
\begin{proof}
  From the decomposition of Tanaka symbol \eqref{Tanaka_decomposition} and involutivity of distributions $\Delta_{\mathcal J}$ it follows that  that for any $i\in [1:k]$ and any section $V$ of the distribution $\displaystyle{\bigoplus_{j\in [1:k]\backslash \{i\}} D_j}$
  \begin{equation}
  \label{infsymm}
  [V, D_i\oplus\Delta_{[1:k]\backslash\{i\}}]
  \subset D_i\oplus\Delta_{[1:k]\backslash\{i\}},
  \end{equation}
   which implies that the local flow $e^{tV}$, generated by the vector field $V$, consists of local symmetries of the distribution  $D_i\oplus\Delta_{[1:k]\backslash\{i\}}$, i.e., 
\begin{equation}
\label{symm}
    (e^{tV})_* \Bigl(D_i\oplus\Delta_{[1:k]\backslash\{i\}} \Bigr)= D_i\oplus\Delta_{[1:k]\backslash\{i\}}. 
\end{equation}
By construction, for any section $V$ of the distribution $\displaystyle{\bigoplus_{j\in [1:k]\backslash\{i\}} D_j}$ (an even of the distribution $\Delta_{[1:k]\backslash\{i\}}$) , we have  
\begin{equation}
\label{piV}
\pi_i^{q_0} \circ e^{tV} = \pi_i^{q_0}.
\end{equation}
This together with \eqref{symm} and the fact that  
\begin{equation}
\label{ker_pi}
\ker (d \pi_i^{q_0})(q) = \Delta_{[1:k]\backslash\{i\}}(q)
\end{equation}
implies that
\begin{equation} 
\label{section_lemma_inter}
d \pi_i^{q_0} \bigl(D_i(e^{tV} q_0) \bigr)=D_i(q_0).
\end{equation}
Consequently, by the definition of $\tilde X_r$ with  $r \in \mathcal{I}_1^i$ given by \eqref{eq:15} we get
\begin{equation} 
\label{section_lemma_inter1}
X_r\bigl(e^{tV} q_0\bigr)\in D_i\bigl(e^{tV} q_0\bigr).
\end{equation}
Finally, by Rashevskii-Chow the point $q_0$ can be connected with any point of $\mathcal F_{\Delta_{[1:k]\backslash\{i\}}} (q_0)\cap U$ by a finite concatenation of integral curves tangent to 
the distribution $\displaystyle{\bigoplus_{j\in [1:k]\backslash\{i\}} D_j}$ and so we can apply the relations \eqref{ker_pi}  and  \eqref{section_lemma_inter1} finite number of times (with corners of the concatenation instead of $q_0$) to get the conclusion of the lemma.
\end{proof}
\begin{Lemma} 
\label{commute_lemma}
If $i\neq j\in [1:k]$, then for every $r\in \mathcal I_i^1$ and $l\in \mathcal I_j^1$ the vector fields $\Tilde X_r$ and $\Tilde X_l$ commute, 
\begin{equation}\label{commute basis}
    [\Tilde{X}_r, \Tilde{X}_l] = 0, \quad \forall r\in \mathcal I_i^1, l\in \mathcal I_j^1, i\neq j.
\end{equation}
\end{Lemma}
\begin{proof}
As before, let $i\in [1:k]$ and $V$ be a section of the distribution $\Delta_{[1:k]\backslash\{i\}}$. 
 Then,  using standard properties of Lie derivatives and 
 \eqref{piV}, one gets
\begin{equation}\label{Lie derivative}
    \begin{split}
         &(\pi_i^{q_0})_* [V, \Tilde{X}_r](q)  = \frac{\mathrm{d}}{\mathrm{d}t} (\pi_i^{q_0})_* (e^{-tV})_* \Tilde{X}_r (e^{tV}(q))\bigg|_{t=0} \stackrel{\eqref{piV}}{=} \\
         &   \frac{\mathrm{d}}{\mathrm{d}t}  (\pi_i^{q_0})_* \Tilde{X}_r (e^{tV}(q))\bigg|_{t=0}
         =  \frac{\mathrm{d}}{\mathrm{d}t} X_r(\pi_i^{q_0}(q)) = 0,
    \end{split}
\end{equation}
where $(\pi_i^{q_0})_*$ denotes the pushforward of the map $\pi_i^{q_0}$. Then, by \eqref{ker_pi},  
the above calculations show that 
\begin{equation}
\label{VTilde}
  [V, \Tilde{X}_r] \in \Delta_{[1:k]\backslash\{i\}}, \quad r\in \mathcal I_i^1,\,\, V \textrm{ is a section of } \Delta_{[1:k]\backslash\{i\}}
\end{equation}
 Using \eqref{VTilde}, first for $V=\Tilde X_l$ such that  $l\in \mathcal I_j^1$ and  $j\neq i$, then switching  roles of $i$ and $j$, and finally using that the distribution $\Delta_{\{i,j\}}$ we get  
 \begin{equation}\label{commute basis_1}
    [\Tilde{X}_r, \Tilde{X}_l] \in \Delta_ {[1:k]\backslash\{i\}}\cap\Delta_{[1:k]\backslash\{j\}}\cap\Delta_{\{i,j\}} = 0, \quad \forall r\in \mathcal I_i^1, l\in \mathcal I_j^1, i\neq j,
\end{equation}
i.e., the vector fields $\Tilde{X}_r$ and $\Tilde{X}_l$ commute.
\end{proof}
 Our final goal is to show the following
\begin {Lemma}
 The frame $(\Tilde X_1, \ldots \Tilde X_m)$ defined by \eqref{eq:15} is $g_1$-orthonormal. 
\end{Lemma}
\begin{proof}
 By Lemma \ref{section_lemma}  the collection $\{\Tilde X_r\}_{r\in \mathcal{I}_i^1}$ is a local frame of $D_i$, hence for every $q\in U$ there exists the $m_i\times m_i$ matrix $T_i(q):=(t^i_{rs}(q))_{r,s\in \mathcal I_i^1} \in \mathrm{GL}(m_i)$, making the transition from the originally chosen local frame   $(X_i)_{i\in \mathcal{I}_i^1}$ of $D_i$ to
 $(\Tilde{X}_i)_{i\in \mathcal{I}_1^1}$, i.e.,
\begin{equation}
\label{T1_trans}
\Tilde{X}_r = \sum_{s\in \mathcal I_i^1} t^i_{rs}X_s,
\quad i\in [1:k], r\in \mathcal I_i^1.
\end{equation}  
The lemma will be proved if we show that  
$T_i\in \mathrm{SO}(m_i)$ for every $i\in [1:k]$. 

First, from the commutativity relation \eqref{commute basis} by direct computations it follows that the entries of the transition matrix-valued function $T_i$  satisfy the following system of partial equations 2: 
    \begin{equation}
    \label{commutativity}
            X_{l}(t^i_{rs})   = -\sum_{v\in \mathcal{I}_i^1} c_{lv}^s t^i_{rv},\quad  \forall r,s\in \mathcal{I}_i^1, l\in [1:m]\backslash \mathcal I_i^1. 
        \end{equation}

Now, using \eqref{commutativity}, we have
\begin{equation}\label{ortho cond 2}
    \begin{split}
        &X_l\left( \sum_{s\in\mathcal I_i^1} t^i_{rs}t^i_{ws}\right) = \sum_{s\in\mathcal I_i^1}\left( X_l(t^i_{rs}) t^i_{ws} + t^i_{rs} X_l (t^i_{ws})\right) \stackrel{\eqref{commutativity}} =\\         &-\sum_{s\in \mathcal{I}_i^1}\sum_{v\in \mathcal{I}_i^1} \left(c_{lv}^s t^i_{rv}t^i_{ws}+ c_{lv}^s t^i_{wv}t^i_{rs} \right)=
         - \sum_{s,v\in \mathcal{I}_i^1}(c_{lv}^s+c_{ls}^v) t^i_{rv}t^i_{ws} = 0,
    \end{split}
\end{equation}
where the second expression of the chain of equalities in the second line of \eqref{ortho cond 2} is obtained by swapping indices $s$ and $v$ in the second term of the first expression of the chain of equalities in the second line of \eqref{ortho cond 2}, and the last equality in  \eqref{ortho cond 2} follows from items (1) and (2) of Proposition \ref{first divi}. 
Therefore, $\sum_{s\in\mathcal I_i^1} t^i_{rs}t^i_{ws}$ is constant on each piece of a leaf of  $\mathcal{F}_{[1:k]\backslash\{i\}}$ lying in $U$. Besides, by \eqref{eq:15}, 
$$
\Tilde{X}_r(q)=X_r(q), \quad \forall r\in \mathcal{I}_i^1, q\in \mathcal{F}_{\{i\}}(q_0)\cap U,
$$ 
so by $g_1$-orthonormality of $\{X_r\}_{r\in \mathcal I_i^1}$ implies that 
$$\sum_{s\in\mathcal I_i^1} t^i_{rs}(q)t^i_{ws}(q)=\delta_{rw}, \quad \forall r,w\in \mathcal I_i^1, q\in \mathcal F_{\{i\}}(q_0)\cap U,$$ 
and hence, by \eqref{ortho cond 2},
$$\sum_{s\in\mathcal I_i^1} t^i_{rs}(q)t^i_{ws}(q)=\delta_{rw}, \quad \forall r,w\in \mathcal I_i^1, q\in U$$ 
This proves $g_1$-orthonormality of $\{\Tilde X_r\}_{r\in \mathcal I_i^1}$.
\end{proof}
We complete the proof of the main Theorem \ref{mainthm} by noticing that if $g_1^i$ denotes the sub-Riemannian metric on the distribution $D_i$ on the leaf $F_i(q_0)$ defined by the condition that $\{X_r\}_{r\in \mathcal I_i^1}$ is orthogonal in this metric, then
$$
(U,D|_U,g_1|_U)=\underset{i=1}{\overset{k}{\Cross}}\Bigl(\mathcal{F}_i(q_0)\cap U, D_i|_{\mathcal{F}_i(q_0)\cap U},g_1^i|_{\mathcal {F}_i(q_0)\cap U}\Bigr). 
$$
\appendix\section{Proof of Proposition \ref{codim3}} \label{appendix A}
Assume by contradiction that $\mathfrak m$ is not $\mathrm{ad}$-surjective. Let $m:=\dim\, \mathfrak m_{-1}$ and $d:=\dim \,\mathfrak m_{-2}$. We start with consideration for general $d$. Assume that 
\begin{equation}
\label{r_def}
r:=\max\limits_{X\in \mathfrak m_{-1}} \mathrm{rank} (\mathrm{ad} X).    
\end{equation}
Then from non-$\mathrm{ad}$-surjectivity  assumption $r<d$.
Take $X_1$ such that 
\begin{equation}
\label{rank_X1}
\mathrm{rank} (\mathrm{ad} X_1)=r.
\end{equation}
Then the rank-nullity theorem implies that 
\begin{equation}
\label{dim_ker_X1}
\dim \,\mathrm{ker} (\mathrm{ad} X_1)=m-r.
\end{equation}
Obviously, $X_1\in \mathrm{ker} (\mathrm{ad} X_1)$. Let us complete it to the basis $(X_1 \ldots X_m)$ of $\mathfrak m_{-1}$ such that  
\begin{equation}
\label{Xr+2}
\mathrm{ker} (\mathrm{ad} X_1)=\mathrm{span}\{X_1, X_{r+2},\ldots, X_m\}.    
\end{equation}
Let 
\begin{equation}
\label{Ys}
Y_l:=[X_1, X_{l+1}], l\in [1:r].
\end{equation}
By  constrictions, $Y_1, \ldots, Y_r$  are linearly independent, and 
\begin{equation}
\label{spanY}
\mathrm{Im} (\mathrm{ad} X_1)=\mathrm {Im} (\mathrm{ad} X_1|_{\mathrm{span}\{X_2, \ldots, X_{r+1}\}})=\mathrm{span}\{Y_1, \ldots, Y_r\}.
\end{equation}

Since $\mathfrak m$ is step $2$  and fundamental , there exist $i<j\in [2, m]$ such that 
\begin{equation}
\label{Xij}
    [X_i, X_j]\notin \mathrm{Im} (\mathrm{ad} X_1).
\end{equation}
Set
\begin{equation}
\label{Y_r+1_def}
Y_{r+1}:=[X_i, X_j].   
\end{equation}
\begin{Lemma}
  The index $j$  (and therefore also $i$) in \eqref{Xij}  does not exceed $r+1$.
\end{Lemma}
\begin{proof}
 Assume by contradiction that $j\geq r+2$. From  maximality of $r$ in \eqref{r_def}, \eqref{rank_X1}, and \eqref{spanY} it follows that  for sufficiently small $t$ 
\begin{equation}
\label{small}
 \mathrm{rank} 
 (\mathrm{ad} (X_1+t X_i|_{\mathrm{span}\{X_2, \ldots, X_{r+1}\}})=r
\end{equation}
 and the spaces $\mathrm{Im} 
 (\mathrm{ad} (X_1+t X_i|_{\mathrm{span}\{X_2, \ldots, X_{r+1}\}})$ are sufficiently closed to $\mathrm{Im} (\mathrm{ad} X_1)$ so that
\begin{equation}
\label{Y_r+1_notin} 
Y_{r+1}\notin \mathrm{Im} 
 (\mathrm{ad} (X_1+t X_i|_{\mathrm{span}\{X_2, \ldots, X_{r+1}\}}).
\end{equation}
On the other hand, from \eqref{Xr+2} and \eqref{Y_r+1_def} it follows that 
\begin{equation}
[X_1+t X_i, X_j]=tY_{r+1}, 
\end{equation}
which implies that $\mathrm{rank} 
 (\mathrm{ad} (X_1+t X_i)>r$ for sufficiently small $t\neq 0$ , This  contradicts the maximality of $r$ in \eqref{r_def} and completes the proof of the lemma.
\end{proof}
 In the proof of the previous lemma, based on \eqref{Xr+2} and \eqref{spanY} we actually have shown that
 \begin{equation}
 \label{ker_m_-1}
 [\mathrm{ker} (\mathrm{ad} X_1),\mathfrak m_{-1}]\subset \mathrm{Im} (\mathrm{ad} X_1). \end{equation}
 After permuting indices we can assume that $(i,j)=(2,3)$, i.e., that 
 \begin{equation}
 \label{X23_cond}
 [X_2, X_3]\notin \mathrm{Im} (\mathrm{ad} X_1).
 \end{equation}
 Now, given $X$ and $\tilde X$ from $\mathfrak m_{-1}$, set 
\begin{equation}
 \label{Lspace}
 L_{X, \tilde X}:=\mathrm{ker} (\mathrm{ad} X)\cap \mathrm{ker} (\mathrm{ad} \tilde X).
\end{equation}
 Let 
\begin{equation}
\label{min_ker_inter}
k:=\min\limits_{X, \tilde X\in\mathfrak m_{-1}} \dim  L_{X, \tilde X}.
\end{equation}
By genericity of \eqref{rank_X1}, \eqref{X23_cond} and \eqref{min_ker_inter} we can choose $X_1$, $X_2$, and $X_3$, maybe after small perturbation,  such that \eqref{rank_X1},  \eqref{X23_cond}, and 
\begin{equation}
\label{ker X21_inter}  
\dim \bigl(\mathrm{ker} (\mathrm{ad} X_1)\cap \mathrm{ker} (\mathrm{ad} X_2)\bigr)=k
\end{equation}
hold simultaneously. 

Now, by item 1 of Proposition \ref{codim3} $d\leq 3$. Therefore either $r=1$ or $r=2$. Consider these two cases separately.

{\bf Case 1: $r=1$.} By \eqref{Xr+2} 
\begin{equation}
\label{r=1_1}
\mathrm{ker} (\mathrm{ad} X_1)=\mathrm{span}\{X_1, X_3,\ldots, X_m\}
\end{equation}
and by this and \eqref{ker_m_-1}
\begin{equation}
\label{r=1_2}
[X_i, X_j]\in \mathrm{Im} (\mathrm{ad} X_1), \quad \forall i\in [2:m], j\in [3, m]
\end{equation}
or, equivalently, from the fundamentality of $\mathfrak m$,  
\begin{equation}
\label{r=1_3}
\mathfrak m_{-2}= \mathrm{Im} (\mathrm{ad} X_1),
\end{equation}
so in fact $d=r$ ($=1$) and this case is done. 

{\bf Case 2: $r=2$.}
By the previous constructions,
\begin{equation}
 \label{X123}
 [X_1, X_2]=Y_1, \quad [X_2, X_3]=Y_3,
\end{equation}
where
\begin{equation}
\label{Y_3_notin}
Y_3\notin \mathrm{Im} (\mathrm{ad} X_1), 
\end{equation}
as a particular case of   \eqref{Y_r+1_notin}
for $r=2$. Then from \eqref{X123} and \eqref{Y_3_notin} by maximality of $r=2$ in  \eqref{r_def},
\begin{equation}
\label{In_ad_X_2}
\mathrm{Im} (\mathrm{ad} X_2) =\mathrm {span}\{Y_1, Y_3\}.  
\end{equation}
From this, \eqref{spanY}, and \eqref{ker_m_-1} it follows that 
\begin{equation}
\label{>4}
[X_2, X_i]\in \mathrm{Im} (\mathrm{ad} X_1)\cap\mathrm{Im} (\mathrm{ad} X_2)=\mathrm{span}\{Y_1\},\quad \forall i\in [4:m].
\end{equation}
Since $[X_1, X_2]=Y_1$ for any $i\in [4:m]$ one can replace $X_i$ by 
\begin{equation}
\label{mod_X_1}
\tilde X_i\equiv X_i \quad \mathrm{mod}\,\, \mathrm{span} \{X_1\}
\end{equation}
such that $[X_2, \tilde X_i]=0$, 
i.e., 
$$\mathrm{ker} (\mathrm{ad} X_2)=\mathrm{span}\{X_2, \tilde X_4, \ldots \tilde X_m\}.$$

This together with \eqref{Xr+2} and \eqref{mod_X_1} implies that
$$
L_{X_1, X_2}=\mathrm{span}\{\tilde X_4, \ldots, \tilde X_m\}$$
Therefore , by \eqref{ker X21_inter}
\begin{equation}
\label{k=m-3}
k=m-3.
\end{equation}
The following Lemma will give a contradiction with item 3 of the assumptions of Proposition \ref{codim3} and therefore will complete the proof of it in the considered case of $r=2$: 
\begin{Lemma}
\label{indep_L_lemma}
The space $L_{X, X_1}$ is the same for all $X\in \mathfrak m_{-1}$ for which $\dim L_{X, X_1}=m-3$ and so, the space $L_{X, X_1}$ lies in the center of $\mathfrak m$ \footnote{The latter conclusion follows from the fact that the set of such $X$ is generic in $\mathfrak m_{-1}$.}.
\end{Lemma}
\begin{proof}
Assume by contradiction that there exist $X_2$ and $X_3$ such that $\dim L_{X_i, X_1}=m-3$, $i=2,3$ but 
\begin{equation}
\label{Lneq}
 L_{X_2, X_1}\neq L_{X_3, X_1}.   
\end{equation}
Obviously, $X_1$, $X_2$ and $X_3$ are linearly independent. By openness of condition \eqref{Lneq} we can always assume that 

\begin{equation}
\label{adneq}
\mathrm{ad} X_1(\mathrm{span} X_2) \neq \mathrm{ad} X_1(\mathrm{span} X_3). 
\end{equation}
We claim that
\begin{equation}
\label{transitivity}
L_{X_2, X_3}=L_{X_2, X_1}\cap L_{X_3, X_1}.
\end{equation}
Before proving \eqref{transitivity}, note that if it holds then by \eqref{Lneq} it will follow that $\dim L_{X_2m X_3}<m-3$ which will contradict the minimality of $k=m-3$ in \eqref{min_ker_inter}.

It remains to prove \eqref{transitivity}.
First, it is clear that 
\begin{equation}
\label{clear_dir}
L_{X_2, X_1}\cap L_{X_3, X_1}=\bigcap_{i=1}^3 \mathrm{ker} (\mathrm{ad} X_i)\subset L_{X_2, X_3}.
\end{equation}
On the other hand, from the dimension assumptions it follows that
\begin{equation}
\label{Ker_ad_X_i}
\mathrm{ker} (\mathrm{ad} X_i)=\mathrm{span}\{X_i\}\oplus L_{X_i, X_1}, \quad i=2,3.
\end{equation}
So if $v\in L_{X_2, X_3}$ then 
\begin{equation}
 \label{v_eq}
 v\equiv \alpha_2 X_2 \,\mathrm{mod}\,\, \mathrm{ker} (\mathrm{ad} X_1) \equiv \alpha_3 X_3 \,\mathrm{mod}\,\, \mathrm{ker} (\mathrm{ad} X_1).
 \end{equation}
Note that \eqref{adneq} means that $X_2$ and $X_3$ are linearly independent modulo $\mathrm{ker} (\mathrm{ad} X_1)$. Hence, \eqref{v_eq} implies that $\alpha_2=\alpha_3=0$, i.e., $v\in \mathrm{ker} (\mathrm{ad} X_1)$. This implies that 
$v\in  L_{X_2, X_1}\cap L_{X_3, X_1}$, i.e.,
$$L_{X_2, X_3}\subset L_{X_2, X_1}\cap L_{X_3, X_1}.$$
This and inclusion \eqref{clear_dir} complete the proof of \eqref{transitivity} and therefore of Lemma \ref{indep_L_lemma}.
\end{proof}

Finally note that in the case of $\dim \mathfrak m_{-2}=4$, even if the assumptions 2 and 3 hold Proposition \ref{codim3} is wrong. Here is the counterexample:

\begin{example}
\label{counter_example}
Let $\mathfrak m=\mathfrak m_{-1}\oplus \mathfrak m_{-2}$ be the step $2$ graded Lie algebra such that
\begin{equation}
\begin{split}
~&\mathfrak m_{-1}=\mathrm{span}\{X_1, \ldots, X_5\}\\
~& \mathfrak m_{-2}=\mathrm{span}\{Y_1, \ldots, Y_4\}
\end{split}
\end{equation}
so that, up to skew-symmetricity, the following brackets of the chosen basis are the only nonzero ones:

\begin{equation}
\begin{split}
~&[X_1, X_i]=Y_{i-1}, \,\,\forall i\in [2:4],\\ 
~&[X_2, X_3]=Y_4, [X_2, X_5]=\beta Y_3,\\   
~&[X_3, X_5]=\delta Y_3, [X_4. X_5]=\lambda Y_3,
\end{split}    
\end{equation}
where $\beta$, $\delta$, $\lambda$ are nonzero constants.
Note that $\mathrm{span} \{X_1, X_2, X_3, Y_1, Y_2, Y_4\}$ is a subalgebra of $\mathfrak m$ isomorphic to the  truncated step $2$ free Lie algebras with three generators and $\mathrm{span} \{X_4, X_5, Y_3\}$ is an ideal of $\mathfrak m$ isomorphic to the $3$-dimensional Heisenberg algebra.  Thus,  \emph{$\mathfrak m$ is a semi-direct sum of  the  truncated  step $2$ free Lie algebras with three generators  and  the $3$-dimensional Heisenberg algebra.}

It can be checked by straightforward computations that here $r$ defined by \eqref{r_def} is equal to $3<d=4$ and that $\mathfrak m_{-1}$ meets the center trivially, i.e., it is indeed a counter-example.

Indeed, if $X\in \mathfrak m_{-1}$, 
\begin{equation}
\label{X_C}
  X=\sum_{i=1}^5 C_i X_i,
\end{equation}
then the map $\mathrm{ad} X$ has the following matrix with respect to the bases $(X_1, \ldots, X_5)$ and $Y_1, \ldots, Y_4)$ of $\mathfrak m_{-1}$ and $\mathfrak m_{-2}$:
\begin{equation}
\label{AdX_matrix}
\mathrm{ad} X=\begin{pmatrix}
-C_2&-C_3&-C_4&0\\
C_1&0&-\beta C_5&C_3\\
0&C_1&\delta C_5&C_2\\
0&0&C_1-\lambda C_3&0\\
0&0&\beta C_2+\delta C_3+\lambda C_4&0
\end{pmatrix}.
\end{equation}
It is easy to check that the maximal rank of this matrix (as a function of C's) is equal to $3$, which implies that $r=3$. Also, this matrix is not equal to zero if $(C_1,\ldots C_5)\neq 0$, which means that $\mathfrak m_{-1}$ meets the center trivially.

Finally note that $\mathfrak m$ is not decomposable. Assuming the converse, i.e., that $\mathfrak m=\mathfrak m^1\oplus \mathfrak m^2$ for some nonzero fundamental graded Lie algebra $\mathfrak m^1$ and $\mathfrak m^2$. Without loss of generality assume that 
\begin{equation}
\label{dim12_neq}
    \dim\, \mathfrak m^1_{-1}\geq \dim\, \mathfrak m^2_{-1}.
\end{equation}
 Since $\mathfrak m_{-1}$ meets the center trivially, it is impossible that $\dim\, \mathfrak m^2_{-1}=1$. Hence by \eqref{dim12_neq} and the fact that $\dim \mathfrak m_{-1}=5$,  we have that  $\dim\, \mathfrak m^2_{-1}=2$ and the algebra $\mathfrak m_{-1}$ is nothing but the $3$-dimensional Heisenberg algebra. Therefore for every nonzero $X\in \mathfrak m^2_{-1}$, the rank of $\mathrm{ad} X$ is equal to $1$. However, it is straightforward to show that if the rank of the matrix \eqref{AdX_matrix} is not greater than $1$, then $(C_1,\ldots C_5)= 0$, which leads to the contradiction. So $\mathfrak m$ is indecomposable.

 An alternative, more conceptual way to prove indecomposability of $\mathfrak m$ is to observe that otherwise, each component in its decomposition will have $-2$ degree part of dimension not greater than $3$ (but not equal to $0$) and by Proposition  \ref{codim3} each component is $\mathrm{ad}$-surjective. Then by Remark \ref{sum_ad_surj}, $\mathfrak m$ is $\mathrm{ad}$-surjective, which is not the case.
\end{example}

\section{Proof of Sublemma \ref{last_last_coro} }\label{appendix B}
Let us derive the expression for $\Tilde b_{m+f}$ with $f\in [1:m_1]$. For the proof of Sublemmma \ref{last_last_coro} we need the case $f=1$ only but the cases of more general $f$ are needed in subsection \ref{D2D3_sec}, in particular in Lemma \ref{4.5_lem_2}.
First, by column operations \eqref{column_op}
\begin{equation}
\label{b^{1,2}_last}
 \Tilde{b}_{m+f}=b^2_f- \sum_{i=1}^k(\alpha_{n_{i-1}+1})^2 \sum_{t=e_{i-1}}^{e_i} (a_{f, m+t})^2 u_{m+t},
\end{equation}
where by \eqref{elem.A} and \eqref{term.of.b.rec.form}  
\begin{align}
~& a^2_{f, m+j}=\vec{h}_1(q_{f,m+j}) + \sum_{r = m+1}^{n}q_{f,r} q_{r,m+j} \label{a^2_{1, m+j}},\\
~&b^2_f= \vec{h}_1(b^1_f) -\sum_{i=1}^k(\alpha_{n_{i-1}+1})^2\sum_{r = m+1}^{n}  \sum_{w\in \mathcal I_i^1}  q_{1r}q_{rw} u_w,
\label{b^2_1}
\end{align}
and $q_{jk}$ and $b^1_f$ are as in \eqref{qjk} and  \eqref{b in first group},
respectively.
Note that from the decomposition of the Tanaka symbol \eqref{Tanaka_decomposition}, it follows that 
\begin{equation}
\label{qjk_0}
q_{wr} = 0,\quad  \forall w\in \mathcal I_i^1, r\in \mathcal I_v^2 \text{ with } i\neq v.
\end{equation}

Substituting \eqref{qjk_0} into \eqref{a^2_{1, m+j}}, we get
\begin{equation}
 \label{a^2_{1, m+j}_0}
a^2_{f, m+t}=\begin{cases}
\vec{h}_1(q_{f,m+t}) + \displaystyle{\sum_{r = m+1}^{m+d_1}q_{fr} q_{r, m+t}},&  t\in [1:d_1]\\
\displaystyle{\sum_{r = m+1}^{m+d_1}q_{fr} q_{r,m+t}},&  t \in [1:n-m]\backslash [1:d_1].
\end{cases}
\end{equation}
Moreover, \eqref{qjk_0} implies that \eqref{b^2_1} can be rewritten as follows:
\begin{equation}
b^2_f= \vec{h}_1(b^1_f) -\sum_{i=1}^k(\alpha_{n_{i-1}+1})^2\sum_{r = m+1}^{m+d_1}  \sum_{v\in \mathcal I_i^1}  q_{fr}q_{rv} u_v.
\label{b^2_1_n}
\end{equation}
By \eqref{b^{1,2}_last}, \eqref{a^2_{1, m+j}_0} , and \eqref{b^2_1_n},  we have 
\begin{equation}
\label{b1,2 step 1}
    \begin{split}
        \Tilde{b}_{m+f} &=  
         \vec{h}_1(b^1_f) -
        \sum_{i=1}^k(\alpha_{n_{i-1}+1})^2\sum_{r = m+1}^{m+d_1}  \sum_{w\in \mathcal I_i^1}  q_{fr}q_{rw} u_w
        \\
        &\quad - \sum_{i=1}^k(\alpha_{n_{i-1}+1})^2 \sum_{t=e_{i-1}+1}^{e_i} (a_{f, m+t})^2 u_{m+t}.
    \end{split}
\end{equation}
Using \eqref{b in first group} and \eqref{huj}, we get 

\begin{small}
\begin{equation}
\begin{split}
\label{hb_1,1}
     &  \Vec{h}_1(b^1_f) = \Vec{h}_1\left( 
         (\alpha_{1})^2 \sum_{r=m+1}^{m+d_1} q_{fr}u_r + \sum _{i=2}^k\left( (\alpha_{1})^2 -(\alpha_{n_{i-1}+1})^2\right) \sum_{w\in\mathcal I_i^1} q_{fw}u_w\right) \\
       & = \alpha_1^2 \sum_{r=m+1}^{m+d_1} \Vec{h}_1(q_{fr})u_r + \alpha_1^2 \sum_{r=m+1}^{m+d_1} q_{fr}\Vec{h}_1(u_r)+ 
       \sum _{i=2}^k\left( (\alpha_{1})^2 -(\alpha_{n_{i-1}+1})^2\right) \sum_{w\in\mathcal I_i^1} \Vec{h}_1(q_{fw})u_w+ \\
        &\quad \sum _{i=2}^k\left( (\alpha_{1})^2 -(\alpha_{n_{i-1}+1})^2\right) \sum_{w\in\mathcal I_i^1} q_{fw}\Vec{h}_1(u_w)
       = \alpha_1^2 \sum_{r=m+1}^{m+d_1} \Vec{h}_1(q_{fr})u_r +   \alpha_1^2\sum_{r=m+1}^{m+d_1} \sum_{x=1}^n q_{fr} q_{rx}u_x+\\
        & \sum _{i=2}^k\left( (\alpha_{1})^2 -(\alpha_{n_{i-1}+1})^2\right) \Bigl(\sum_{w\in\mathcal I_i^1} \Bigl(\Vec{h}_1(q_{fw})u_w+
          \sum_{x=1}^n q_{fw}q_{wx}u_x\Bigr) \Bigr).
\end{split}
\end{equation}
\end{small}
Substituting \eqref{a^2_{1, m+j}_0}  and \eqref{hb_1,1} into \eqref{b1,2 step 1}, and taking  into account \eqref{a^2_{1, m+j}_0} again, we get the following cancellations: 
\begin{equation}
\label{b1,2 step 2}
    \begin{split}
         &\Tilde{b}_{m+f} = \cancel{\alpha_1^2 \sum_{r=m+1}^{m+d_1} \Vec{h}_1(q_{fr})u_r} +   \alpha_1^2\sum_{r=m+1}^{m+d_1} \sum_{w=1}^{n}q_{fr} q_{rw}u_w+\\
        & \sum _{i=2}^k\left( (\alpha_{1})^2 -(\alpha_{n_{i-1}+1})^2\right) \Bigl(\sum_{w\in\mathcal I_i^1} \Bigl(\Vec{h}_1(q_{fw})u_w+
          \sum_{x=1}^n q_{fw}q_{wx}u_x\Bigr) \Bigr)
       \\
        & 
        \sum_{i=1}^k(\alpha_{n_{i-1}+1})^2\sum_{r = m+1}^{m+d_1}  \sum_{w\in \mathcal I_i^1}  q_{fr}q_{rw} u_w-\cancel{(\alpha_{1})^2 \sum_{t=1}^{d_1} \vec{h}_1(q_{fm+t}) u_{m+t}} -\\
        &
        \sum_{i=1}^k(\alpha_{n_{i-1}+1})^2 \sum_{t=e_{i-1}+1}^{e_i} \sum_{r = m+1}^{m+d_1}q_{fr} q_{r, m+t} u_{m+t}.
    \end{split}
\end{equation}
Applying the relation $[1:n]=\displaystyle{\bigcup_{i=1}^k \mathcal I_i^1\cup \mathcal I_i^2 }$ to the sum in the second term of \eqref{b1,2 step 2}, we get
\begin{equation}
\label{b1,2 step 2_fin}
\begin{split}
 & \Tilde{b}_{m+f} =\sum _{i=2}^k\left( (\alpha_{1})^2 -(\alpha_{n_{i-1}+1})^2\right)\Bigl(\sum_{w\in\mathcal I_i^1} \Bigl(\Vec{h}_1(q_{fw})u_w+
          \sum_{x=1}^n q_{fw}q_{wx}u_x\Bigr) +\\&\sum_{r = m+1}^{m+d_1}  \sum_{w\in \mathcal I_i^1}   q_{fr} q_{rw} u_w+\sum_{r = m+1}^{m+d_1}\sum_{t\in \mathcal I_i^2}q_{fr} q_{rt} u_t\Bigr).
\end{split}
\end{equation}
From now on let $f=1$. Since by \eqref{qjk} all $q_{jk}$ depend on $u_i$'s with $i\in [1:m]$ only, from expression \eqref{b1,2 step 2_fin}  the variable $u_j$ with $j\in \mathcal I_1^2$  may appear only in the following terms of \eqref{b1,2 step 2_fin}:
\begin{equation}
\label{H2.1}
\sum _{i=2}^k\left( (\alpha_{1})^2 -(\alpha_{n_{i-1}+1})^2\right)\sum_{w\in\mathcal I_i^1} \bigl(\Vec{h}_1(q_{1w})u_w+q_{1w}q_{wj}u_j\bigr) ).
\end{equation}
 Moreover,   $q_{wj} =0$ for $w\in [1:m]\backslash \mathcal I_1^1$ and $j\in \mathcal I_1^2$ by \eqref{qjk_0}, so  the variable $u_j$ with $j\in \mathcal I_1^2$  may appear only in the following terms of \eqref{b1,2 step 2_fin}: 
 \begin{equation}
\label{H2.2}
\sum _{i=2}^k\left( (\alpha_{1})^2 -(\alpha_{n_{i-1}+1})^2\right)\sum_{w\in\mathcal I_i^1} \Vec{h}_1(q_{1w})u_w
\end{equation}
or, more precisely, in terms $\Vec{h}_1(q_{1w})$ with $w\in[1:m]\backslash \mathcal I_1^1$, $i\in[2:k]$. Let us analyze these terms. Using \eqref{qjk} and \eqref{huj} we get 
\begin{equation}
\label{hq}
    \begin{split}
        \Vec{h}_1(q_{1w}) &= \sum_{r=1}^m \left( c^w_{r1}\Vec{h}_1(u_r) + \Vec{h}_1(c^w_{r1})u_r  \right) \\
        &= \sum_{r=1}^m \sum_{x=1}^n c^w_{r1} q_{rx}u_x + \sum_{r=1}^m \Vec{h}_1(c^w_{r1})u_r.
    \end{split}
\end{equation}
The second term in \eqref{hq}
does not depend on  $u_j$ with $j\in \mathcal I_1^2$,  while to get  this $u_j$ in the first term the index $x$ must be equal to $j$.  So, the terms containing $u_{j}$ in $\Vec{h}_1(q_{1w})$  are
\begin{equation}
\label{term11}
    u_{j}\sum_{r=1}^m  c^w_{r1} q_{rj}.
\end{equation}
Recall that
\begin{equation}
\label{cases_fin}
    \begin{cases}
         c^w_{r1} = 0, & \text{ if } r\in \mathcal I_1^1
         \,\,w\in[1:m]\backslash \mathcal I_1^1;\\
         q_{rj} = 0, & \text{ if } r\in [1:m]\backslash \mathcal I_1^1.
    \end{cases}
\end{equation}
Here the first line comes from  \eqref{c^k} and the second line comes from \eqref{qjk_0}.
So, plugging \eqref{cases_fin} into \eqref{term11} we get that $u_{j}$ does not appear in $\Tilde{b}_{m+1}$ and the proof of sublemma \ref{last_last_coro} is completed.

\end{document}